\documentclass[11pt, final]{amsart}
\usepackage{amsmath,amsthm,amssymb,amsfonts}
\usepackage{graphicx}
\usepackage{xspace}
\usepackage{xcolor}
\usepackage[colorlinks=false, final]{hyperref}
\usepackage[obeyDraft, textsize=footnotesize, colorinlistoftodos]{todonotes}
\usepackage{verbatim}
\usepackage[scr]{rsfso} 
\usepackage{enumitem}
\setlist[enumerate,1]{label=(\arabic*)} 
\usepackage{multicol}
\usepackage{geometry}

\newcommand{\EndA}[1][\TXn]{\ensuremath{\mathrm{End}(#1)}\xspace}

\newcommand{\End}{\ensuremath{\mathrm{End}}\xspace}
\newcommand{\AutA}[1][\TXn]{\ensuremath{\mathrm{Aut}(#1)}\xspace}
\newcommand{\TXn}[1][n]{\mathcal{T}_{#1}\xspace}
\newcommand{\Sn}[1][n]{\mathcal{S}_{#1}\xspace}
\newcommand{\An}[1][n]{\mathcal{A}_{#1}\xspace}
\newcommand{\Klein}{\mathcal{K}\xspace}
\newcommand{\EndTn}[1][n]{\mathcal{E}_{#1}\xspace}
\newcommand{\AutTn}[1][n]{\mathcal{G}_{#1}\xspace}
\newcommand{\singn}{\EndTn\priv \AutTn}
\newcommand{\pte}[1][t,e]{\phi_{#1}}
\newcommand{\ptce}{\pte[t^2, e]}
\newcommand{\puf}{\pte[u,f]}

\newcommand{\phid}[1][\id]{\pte[\id,#1]}
\newcommand{\pee}[1][e]{\pte[#1,#1]}
\newcommand{\ptctc}[1][t]{\phi_{ {#1}^2, {#1}^2 }}
\newcommand{\psid}{\varepsilon}
\newcommand{\Fxe}[1][e]{\mathrm{Fix}(#1)\xspace}
\newcommand{\Fxte}{\Fxe[t,e]}
\newcommand{\perte}[1][t,e]{\ensuremath{(#1)}\xspace}
\newcommand{\set}[1]{\ensuremath{\left\lbrace #1 \right\rbrace}\xspace}
\newcommand{\pmap}[1]{\begin{pmatrix} #1 \end{pmatrix}}

\theoremstyle{plain} 
\newtheorem{thm}{Theorem}[section] 
 
\newtheorem{prop}[thm]{Proposition} 
\newtheorem{proposition}[thm]{Proposition}
\newtheorem{lem}[thm]{Lemma}
\newtheorem{lemma}[thm]{Lemma}
\newtheorem{cor}[thm]{Corollary}
\newtheorem{corollary}[thm]{Corollary}

\newcounter{scount}[subsection] 

\theoremstyle{definition} 
\newtheorem{defn}[thm]{Definition} 
\newtheorem{definition}[thm]{Definition} 
\theoremstyle{remark} 
\newtheorem{rem}[thm]{Remark} 
\newtheorem{remark}[thm]{Remark} 
\newtheorem*{rem*}{Remark} 

\newtheorem*{notation*}{Notation}

\newtheorem{example}[thm]{Example}

\counterwithin{equation}{section}
\counterwithin{figure}{section}

\newcommand{\GreenL}{\mathscr{L}}
\newcommand{\GreenR}{\mathscr{R}}
\newcommand{\GreenD}{\mathscr{D}}
\newcommand{\GreenJ}{\mathscr{J}}
\newcommand{\GreenH}{\mathscr{H}}

\newcommand{\GL}[1][]{\ensuremath{\mathop{\GreenL_{#1}}}\xspace}
\newcommand{\GR}[1][]{\ensuremath{\mathop{\GreenR_{#1}}}\xspace}
\newcommand{\GD}[1][]{\ensuremath{\mathop{\GreenD_{#1}}}\xspace}
\newcommand{\GH}[1][]{\ensuremath{\mathop{\GreenH_{#1}}}\xspace}
\newcommand{\GJ}[1][]{\ensuremath{\mathop{\GreenJ_{#1}}}\xspace}
\newcommand{\GLs}[1][]{\ensuremath{\mathop{\GreenL^*_{#1}}}\xspace}
\newcommand{\GRs}[1][]{\ensuremath{\mathop{\GreenR^*_{#1}}}\xspace}
\newcommand{\GDs}[1][]{\ensuremath{\mathop{\GreenD^*_{#1}}}\xspace}
\newcommand{\GJs}[1][]{\ensuremath{\mathop{\GreenJ^*_{#1}}}\xspace}
\newcommand{\GLt}[1][]{\ensuremath{\mathop{\widetilde{\GreenL}_{#1}}}\xspace}
\newcommand{\GRt}[1][]{\ensuremath{\mathop{\widetilde{\GreenR}_{#1}}}\xspace}
\newcommand{\GDt}[1][]{\ensuremath{\mathop{\widetilde{\GreenD}_{#1}}}\xspace}
\newcommand{\GJt}[1][]{\ensuremath{\mathop{\widetilde{\GreenJ}_{#1}}}\xspace}
\newcommand{\GHs}{\ensuremath{\mathop{\GreenH^*}}\xspace}
\newcommand{\GHt}{\ensuremath{\mathop{\widetilde{\GreenH}}}\xspace}

\newcommand{\GLtcGRt}{\mathop{\widetilde{\GreenL}\!\circ\!\widetilde{\GreenR}}}

\newcommand{\im}{\mathrm{im\,}}
\newcommand{\id}{\mathrm{id}}
\newcommand{\Card}[1]{\lvert #1 \rvert}
\newcommand{\priv}{\setminus}

\newcommand{\ssi}{\Leftrightarrow}

\def\mathcolor#1#{\@mathcolor{#1}}
\def\@mathcolor#1#2#3{%
  \protect\leavevmode
  \begingroup\color#1{#2}#3\endgroup
}

\definecolor{cadetblue}{rgb}{0.37, 0.62, 0.63} 
\definecolor{celadon}{rgb}{0.67, 0.88, 0.69} 
\definecolor{arylide}{rgb}{0.91, 0.84, 0.42} 

\makeatletter
\providecommand\@dotsep{5}
\makeatother

\title[The structure of End($\TXn$)]{The structure of END(\boldmath{$\TXn$})}

\author[V. Gould]{Victoria Gould}
\address{University of York}
\email{victoria.gould@york.ac.uk}
\author[A. Grau]{Ambroise Grau}
\address{University of York}
\email{ambroise.grau@york.ac.uk}
\author[M. Johnson]{Marianne Johnson}
\address{University of Manchester}
\email{Marianne.Johnson@manchester.ac.uk}

\begin{document}

\keywords{Full transformation semigroup; endomorphisms; ideals; presentations}

\subjclass[2020]{Primary: 20M05, 20M10, 20M20}
\maketitle

\begin{abstract}
The full transformation  semigroups $\TXn$, where $n\in \mathbb{N}$,  consisting of all maps from a set of cardinality $n$ to itself, are arguably the most important family of finite semigroups.  This article investigates the endomorphism monoid $\EndA$ of $\TXn$.  The determination
of the  {\em elements} of $\EndA$ is due  Schein and
Teclezghi.  Surprisingly, the {\em algebraic structure} of $\EndA$ has not been further
explored.  We describe Green’s relations and extended Green's relations on $\EndA$, and the generalised regularity properties of these monoids. In particular, we prove that $\GH=\GL \subseteq \GR=\GD=\GJ$ (with equality if and only if $n=1$); the idempotents of $\EndA$ form a band (which is equal to $\EndA$ if and only if $n=1$) and  also the regular elements of $\EndA$ form a subsemigroup (which is equal to $\EndA$ if and only if $n\leq 2$). Further, the regular elements of $\EndA$ are precisely the idempotents together with all endomorphisms of rank greater than $3$. We also provide a presentation for $\EndA$ with respect to a minimal generating set.
\end{abstract}

\section{Introduction}
\label{sec:intro}
The full transformation semigroup $\TXn$ on a finite set $\{1, \ldots, n\}$ is an important object in algebra.  It is therefore natural to study its endomorphism monoid $\EndA[\TXn]$.  The elements of $\EndA[\TXn]$ were described by Schein and Teclezghi~\cite{ST98}, and we use this description as the starting point for our investigation of the algebraic structure and properties of $\EndA[\TXn]$.  Our initial task is to partition $\EndA[\TXn]$ in several ways, depending on the {\em rank} of an element (the cardinality of its image) and  {\em type} of an element (a quality of the image that determines the behaviour of 
the element in products).  This partitioning is then made use of throughout the text.

One of the first questions to ask about a monoid $S$ is whether it is {\em regular}, that is, for any $a\in S$ there is a $b\in S$ such that
$a=aba$. Regularity is intimately connected with the position and nature of the {\em idempotents} in the monoid, where
$e\in S$ is {\em idempotent} if $e=e^2$.  We show that for $n\geq 3$ the monoid $\EndA[\TXn]$ is not regular. On the other hand the set of idempotents 
of $\EndA[\TXn]$ is very special in that it is a {\em band}, that is, a semigroup of idempotents. Moreover it is a special kind of band, namely a 
{\em left regular band}. 

We  next investigate minimal generating sets of $\EndA[\TXn]$, and provide a presentation for $\EndA[\TXn]$ in terms of these generators.

A major tool used in understanding monoids is that of {\em Green's relations} $\GR,\GL, \GH,\GD$ and $\GJ$, which measure mutual divisibility properties between elements by considering the principal (one-sided) ideals that they generate. Specifically, for elements $a,b$ of a monoid $S$ we have
$a\GR b$ if and only if $aS=bS$; this is equivalent to $a=bs$ and $b=at$ for some $s,t\in S$. The relation $\GL$ is defined dually, and
$\GJ$ is defined by $a\GJ b$ if and only if $SaS=SbS$. Finally, 
$\GH=\GR\wedge \GL=\GR\cap \GL$ and $\GD=\GR\vee \GL=\GR\circ \GL$ with $\GD=\GJ$ for a finite semigroup; the latter two assertions are  theorems -  see, for example, \cite{howie:1995}.  We characterise Green's relations on $\EndA[\TXn]$.  In particular, we show that $\GH=\GL \subseteq \GR=\GD=\GJ$, and for $n\neq 4$ the relation $\GL$ is trivial outside of the group of units. We  then use the description of $\GJ$   to fully determine the  ideals of $\EndA$.

Green's relations provide an alternative characterisation of regularity, specifically, a  monoid is regular if every $\GR$-class (or, equivalently,  every $\GL$-class) contains an idempotent. 
We have commented that $\EndA[\TXn]$ is not regular, however, we show that  it satisfies the weaker property of being left abundant, where a monoid is {\em left abundant}  if every class of the extended Green's relation $\GRs$ contains an idempotent. 
This follows from the fact that the relation $\GRs$ on $\EndTn$ precisely captures the relation of having the same rank. We postpone giving the definitions of the extended Green's relations $\GRs, \GLs,\GHs,\GDs$ and $\GJs$ until later in the text, but suffice to say here they are relations of mutual cancellativity. They first arose in the work of Pastijn \cite{pastijn:1975}
and McAlister \cite{mcalister:1976}. As observed in \cite{fountain:1977}, it follows from \cite{kilp:1973} that a monoid is left abundant precisely when   every monogenic right act is projective.  We determine the {\em $*$-relations} $\GRs, \GLs,\GHs,\GDs$ and $\GJs$ on $\EndA[\TXn]$, from which we observe that 
$\EndA[\TXn]$ does not satisfy the dual property of being right abundant. To complete the picture of the extended Green's relations we consider the {\em $\sim$-relations}
 $\GRt, \GLt,\GHt,\GDt$ and $\GJt$ which extend the corresponding $*$-relations. They are formulated using idempotent left identities: again, we postpone their definition until later. The $\sim$-relations 
 first appear in \cite{elqallali:1980} and have subsequently proved to be useful in a number of ways. For example, they are inherent in the characterisation of left restriction semigroups \cite{hollings:2009}  and in determining varieties containing quasi-varieties of 
 abundant semigroups.

The main focus of our work will be on the general behaviour of $\EndA[\TXn]$ which emerges for $n \geq 5$. The cases $n \leq 4$ exhibit several degenerate or exceptional behaviours: for instance, it is immediate that $\EndA[\mathcal{T}_1] = \AutA[\mathcal{T}_1]$ is the trivial group, and 
we will see that $\EndA[\mathcal{T}_4]$ is unique in that it contains endomorphisms of rank $7$. 

 The structure of the paper is as follows. In Section \ref{sec:lowrank} we remind the reader of the description of the elements of $\EndTn$ from \cite{ST98}.
 We   prove several fundamental results concerning the non-automorphisms in $\EndA[\TXn]$ and give the partitions promised above.
Sections \ref{sec:idempotent}--\ref{sec:extended Greens}  consider the structure of $\EndA[\TXn]$ for $n \geq 5$:   in Section~\ref{sec:idempotent} we describe some properties of the idempotents and determine the regular elements; in Section \ref{sec:gens} we provide a minimal generating set and describe a presentation
for $\EndA[\TXn]$ in terms of these generators;  in  Section \ref{sec:Green} we give a description of Green's relations and use this to determine the ideal structure of
$\EndA[\TXn]$; we consider the extended Green's relations on $\EndA[\TXn]$ in  Section \ref{sec:extended Greens}.  To complete the picture, in Section \ref{sec:degenerate}, we  analyse the structure of $\EndA[\TXn]$ for $n \leq 4$. Finally, in Section~\ref{sec:related}, we indicate how the ideas in this article could be developed and extended.

\textbf{Notation and conventions:} Throughout the paper we write $\Sn\subseteq \TXn$ to denote the symmetric group and $\An\subseteq \Sn$ the alternating group, that is, the subgroup of $\Sn$ of all even permutations. By a slight abuse of notation, we suppress the dependence on $n$ and write simply $\id$ to denote the identity element of $\TXn$. For an element $g\in \Sn$, and $s\in\TXn$, we denote by $s^g$ the product $g^{-1}sg$. For $1\leq i \neq j \leq n$ we also write $(i \,j)$ for the transposition swapping $i$ and $j$ in $\Sn$, and $c_i$ for the constant map with image $i$ in $\TXn$. For $n \neq 4$, the alternating group is the only non-trivial proper normal subgroup of $\Sn$, whilst for $n=4$ there is one additional normal subgroup $\Klein = \{\id, (1\,2)(3\,4), (1\,3)(2\,4), (1\,4)(2\,3)\}$. A straightforward calculation reveals that for any $t \in \Sn[4]$ there is a unique element of $\Klein t$ which fixes $4$. 
 To facilitate readability, we  take the convention that elements of $\TXn$ will be written using Roman letters, while endomorphisms of $\TXn$ will be written using Greek letters. We will also use the short-hand notation $\EndTn=\EndA[\TXn]$ and $\AutTn=\mathrm{Aut}(\TXn)$ for the endomorphism monoid and automorphism group of $\TXn$. The identity element of $\EndTn$ (and hence also $\AutTn$) is the trivial automorphism, which will be denoted by  $\psid$ (again, suppressing the dependence on $n$ where convenient).

We attempt to keep the exposition as self-contained as possible, but refer the reader to \cite{howie:1995} for further details of semigroup notions we employ.

\section{Singular endomorphisms}
\label{sec:lowrank}

In this section we record several useful properties of the endomorphisms in  $\EndTn \priv \AutTn$ of $\EndA[\TXn]$, which we refer to as {\em singular endomorphisms}. We begin by recalling the characterisation of the   semigroup endomorphisms of $\TXn$ (that is, endomorphisms preserving the binary operation, but not necessarily the identity)  due to  Schein and Teclezghi~\cite{ST98}.

\begin{thm}[\cite{ST98}]\label{thm:endomorphisms}
Let $g\in \Sn$ and $t,e\in \TXn$ and define maps $\psi_g$ and $\pte$ for any $s\in \TXn$ by:
$$s\psi_g = s^g \quad \text{and} \quad 
s\phi_{t,e} = \begin{cases} t &\mbox{if }s\in \Sn\priv \An, \\ t^2 &\mbox{if }s\in \An, \\ e &\mbox{if }s\in \TXn\priv \Sn. \end{cases}$$
Then the automorphisms and endomorphisms of $\TXn$ are described as follows.
\begin{enumerate}[left=10pt]
\item $\AutTn= \{\psi_g: g \in \Sn\}$.
\item For $n \neq 4$, 
$$\EndTn= \AutTn \cup \{\pte: t,e \in \TXn, \, t^3=t,\;  te=et=e^2=e\}.$$
\item For $n=4$
$$\EndTn[4] = \AutTn[4] \cup \{\pte: t,e \in \TXn[4], \, t^3=t,\;  te=et=e^2=e\} \cup \{\sigma^g: g \in \Sn[4]\},$$ 
where $s\sigma = c_4$ if $s\in \TXn[4]\priv \Sn[4]$ and $s\sigma = p_s$ if $s\in \Sn[4]$, where $p_s$ denotes the unique element in the coset of $\Klein s$ 
which fixes $4$, and for all $s \in \TXn[4]$ and all $g\in \mathcal{S}_4$ $s\sigma^g = (s\sigma)^g$.
\end{enumerate}
\end{thm}

We remark that in the case $n=4$ in Theorem~\ref{thm:endomorphisms} there is no bias towards 4; any apparent bias is dealt with by the process of conjugation. Note also that $\sigma=\sigma^{\id}$ and 
$\psi_{\id}=\psid$.

It is easy to see that for any $g,h\in \mathcal{S}_n$ we have that $\psi_g\psi_h = \psi_{gh}$. Further if $\psi_g = \psi_h$ then an easy computation gives that for any $i\in \{ 1,\ldots, n\}$ we have that
$c_{ig}=c_i\psi_g=c_i\psi_h=c_{ih}$, so that $ig=ih$ and, since $i$ was arbitrary, $g=h$. It follows that $\AutTn$ is isomorphic to the symmetric group $\Sn$. For $\alpha \in \EndTn$ we write $\im\alpha$ to denote the image of $\alpha$, and define the {\em rank} of $\alpha$ to be the cardinality of $\im\alpha$. 
Clearly the image of each automorphism $\psi_g$ is the whole of $\TXn$, and hence has rank $\Card{\TXn}=n^n$. The remaining endomorphisms, that is, the singular endomorphisms,  have rank strictly less than $n^n$. From the above theorem the rank of each singular endomorphism is either $1, 2,3,$ or $7$. 
For, the image of each endomorphism of the form $\pte$ is the set $\set{t,t^2,e}$, which can have up to three distinct elements. In the case where $n=4$ an easy computation gives that 
$\im\sigma^g = \{t^g: t \in \Sn[4], 4t=4\} \cup \{c_{4g}\},$ has precisely $7$ elements.

\subsection{Endomorphisms of rank at most 3}
In order to encapsulate the conditions of the elements $t$ and $e$ so that $\pte$ is an endomorphism, we let
$$U_n =\set{t\in\TXn \mid t^3 = t \text{ with } te=et=e^2=e \text{ for some } e\in\TXn}.$$

We say that $\perte$ form a \emph{permissible pair}, if $\pte\in \EndTn$, and we denote by $P_n$ the set of all permissible pairs, that is,
$$ P_n = \set{(t,e)\mid t\in U_n, \, te = et = e^2 = e}.$$

Before considering the endomorphisms in $\EndTn$ further, we give some important properties of the sets $U_n$ and $P_n$.

\begin{lemma}
\label{lem:counting}
The set $U_n$ consists of all elements $t=t^3$ satisfying $kt=k$ for at least one $1 \leq k \leq n$. Moreover, for $t \in U_n$ the number of permissible pairs with first component equal to $t$ is $$\sum_{r=1}^{|J|} \binom{|J|}{r}r^{|I|+|J|-r}$$
where $J = \{k: kt=k\}$ and $I$ is maximal such that $t$ restricts to a fixed point free permutation on $I \cup It$.
\end{lemma}

\begin{proof}
Let $t \in U_n$. By definition we have $t^3=t$. Note that if there exists $e \in \TXn$ such that $et=e$, then then for all $k \in \{1, \ldots, n\}$ we must have $ket = ke$, giving that all elements in the image of $e$ are fixed by $t$. 
Conversely, suppose that $t^3=t$  and $kt=k$. Then it is easy to see that $c_kt=tc_k=c_k=c_k^2$, and hence $t \in U_n$.

Note that if $t \in \TXn$ satisfies $t^3=t$ then for all $k \in \{1, \ldots, n\}$ there exist $i_k, j_k$ such that $kt=j_k$, $kt^2=j_kt=i_k$ and $kt^3=i_kt=j_kt^2=kt$. 
If $j_k=k$ then also $i_k=k$. It follows that  $J = \{k: i_k=j_k=k\}$, $K = \{k: k \neq j_k =i_k\}$, $L= \{k: k=i_k \neq j_k\}$ and $M = \{k: k \neq j_k \neq i_k \neq k\}$,  partition the domain 
$\{ 1,\cdots, n\}$ of $t$.   It is clear from these definitions that $t$ restricts to the identity on $J$, $Kt \subseteq J$ and $Mt \subseteq L$. Moreover, $L$ is maximal such that $t$ restricts to a fixed-point free permutation of order $2$ on $L$; let us partition $L = I \cup It$. 

The general picture to have in mind  is as follows. For each transformation $t \in \TXn$ satisfying $t^3=t$, we may partition the domain of $t$ as:
\[
\begin{array}{rclc}
\{1, \ldots, n\} &=& J \cup K \cup L \cup M \mbox{ where }& \\
J &=& \{j: jt=jt^2=j\} & \begin{minipage}{4cm}\centering\begin{tikzpicture}
\node (a) at  (0,0) {$j$};
\path (a) edge [loop right] node {} (a);
\end{tikzpicture}\end{minipage}\\[.3ex]
K &=& \{k: kt= kt^2 \neq k\}& \begin{minipage}{4cm}\centering\begin{tikzpicture}
\node (a) at  (0,0) {$k$};
\node (b) at  (1.5,0) {$kt$};
\path[->] (a) edge node {} (b);
\path (b) edge [loop right] node {} (b);
\end{tikzpicture}\end{minipage}\\
L &=& \{l: lt\neq lt^2=l \}&  \begin{minipage}{4cm}\centering\begin{tikzpicture}
\node (a) at  (0,0) {$l=lt^2$};
\node (b) at  (2,0) {$lt\phantom{t^2}$};
\path[->] (a) edge [bend  right] node {} (b);
\path[->] (b) edge [bend  right] node {} (a);
\end{tikzpicture}\end{minipage}\\
M &=& \{m: m \neq mt \neq mt^2 \neq m\} &  \begin{minipage}{4cm}\centering\begin{tikzpicture}
\node (a) at  (0,0) {\phantom{t}$m$};
\node (b) at  (1.3,0) {$mt$};
\node (c) at  (3,0) {$mt^2$};
\path[->] (a) edge node {} (b);
\path[->] (b) edge [bend  right] node {} (c);
\path[->] (c) edge [bend  right] node {} (b);
\end{tikzpicture}\end{minipage}
\end{array}\]
Noting that $l \in L$ if and only if $lt \in L$, it is clear that $L$ may be further partitioned into two sets of equal 
 size, $I = \{i_1, \ldots, i_s\}$ and $It = \{i_1t,\ldots, i_st\}$. Thus we may write each $t$ such that $t^3=t$ in the form:
\begin{eqnarray*}
t = \pmap{\cdots j \cdots & \cdots k \cdots & \cdots i_r \cdots & \cdots i_rt  \cdots & \cdots m \cdots \\ \cdots j \cdots & \cdots kt \cdots & \cdots i_rt \cdots  &\cdots i_r\cdots & \cdots mt \cdots}\!, \mbox{ with }kt \in J \mbox{ and }mt \in I \cup It,
\end{eqnarray*}
where here $j$, $k$ and $m$ denote arbitrary elements of the sets $J$, $K$ and $M$ respectively. We claim that the elements $e$ satisfying $e=e^2=te=et$ are precisely those of the form:
\begin{eqnarray*}
	e = \pmap{\cdots j \cdots &\cdots k \cdots &\cdots i_r \cdots&\cdots i_rt \cdots &\cdots m\cdots \\ \cdots jf \cdots & \cdots ktf \cdots & \cdots i_rf \cdots &\cdots i_rf\cdots &\cdots mtf'\cdots}\!,
\end{eqnarray*}	
where $f$ is any function $f: J \cup I \rightarrow R$ fixing a non-empty subset $R \subseteq J$ pointwise, and  $f': I \cup It \rightarrow R$ is defined by $i_rtf'=i_rf'=i_rf$.
(For example, given the transformation
$t =t^3= \pmap{1 & 2& 3& 4& 5\\1 & 3& 2& 1& 2\\}\!,$ this description yields that 
$e = \pmap{1 & 2& 3& 4& 5\\1 & 1& 1& 1& 1\\}$ is the unique idempotent satisfying $te=et=e$. Indeed, if $t$ fixes exactly one element then we have no choice but to take $R=J$, and since this is a singleton set this leaves no choice for the function $f$.)

Now to prove the claim. 
If $t \in U_n$, then by the first paragraph of the proof we may assume that $|J|>0$. It is then straightforward to describe the elements $e$ which satisfy $et=te=e=e^2$. The condition that $e$ is idempotent implies that $e$ fixes each element of its image. Further,  if $R$ is a subset of $\{1,\ldots ,n\}$, then there is a bijection between
functions $f:\{1,\ldots ,n\}\setminus R\rightarrow R$ and idempotents with image equal to $R$.  Note that (as observed above) the condition $et=e$ implies that the image of $e$ is contained in $J$.   Let $R$ be a non-empty subset of $J$. We claim that every function $f: (J\priv R) \cup I \rightarrow R$ extends uniquely to the whole of  $\{ 1,\ldots, n\}$  to give an idempotent $e \in \TXn$ with image $R$ satisfying the constraints $et=te=e$.

We have observed that every idempotent must fix its image. Thus, if $e$ is to be an idempotent with image $R$ extending $f$, we have no choice on how $e$ must act on $J \cup I$.  Now the condition $te=e$ forces $ke=kte$ for all $k \in \{1, \ldots ,n\}=J\cup K \cup I \cup It\cup M$. If $k \in K$ then $kt \in J$ and $ke =kte = (kt)f$. If $k \in It$, then $kt \in I$ and the condition $e=te$ forces $ke=kte=(kt)f$. 
If $k \in M$ then $kt \in I \cup It$  and hence either $kt\in I$ and $ke=kte=(kt)f$, or 
$kt \in It$.  In the latter case,  certainly $kt^2\in I$, so that $kt=kt^3=i_kt$ for $i_k=kt^2\in I$, and then $ke=kte=i_kf$. 
Since the sets $J \cup K \cup I \cup It \cup M$ partition the domain of $t$, this shows that there is exactly one way to extend $f$ to an idempotent $e \in \TXn$ with image $R$ satisfying $et=te=e$.

For a fixed $t$, the number of idempotents such that $(t,e) \in P_n$ is therefore found by summing the total number of functions $f: (J \priv R) \cup I \rightarrow R$ as $R$ ranges over non-empty subsets of $J$. That is, for $t \in U_n$ with partition as given above we have that the number of idempotents $e$ satisfying $(t,e) \in P_n$ is
$\sum_{r=1}^{|J|} \binom{|J|}{r}r^{|I|+|J|-r}$.
\end{proof}

We now gather together some routine but useful facts concerning the set $P_n$.

\begin{lem} \label{lem:facts on U}
Let $k \in \mathbb{N}$, $t,e \in \TXn$ and $g \in \Sn$.
\begin{enumerate}
\item If $t$ is idempotent, then $t\in U_n$ and $(t,t) \in P_n$.\label{enum:fact idempotents in Un and Pn}
\item If $(t,e) \in P_n$ with $t^2=e$, then $t=e$.\label{enum:fact on t2 is e forces t is e}
\item We have $t\in \Sn\cap U_n$ if and only if $t^2 = \id$. Consequently, $t\in\Sn\cap U_n$ is a product of an odd number [resp. even number] no more than $(n-1)/2$ of disjoint transpositions if $t\in\Sn\priv\An$ [resp. if $t\in\An$]. \label{enum:fact on t2 in Sn is id}
\item We have $\left(t^g\right)^k = t^g$ if and only if $t^k = t$. Additionally, $t^g = \id$ if and only if $t=\id$.\label{enum:fact on conjugate of power}
\item If $(t,e)\in P_n$, then $(t^g,e^g)\in P_n$. \label{enum:fact on conjugated pair}
\item If $\perte\in P_n$ and $e=\id$, then $t=\id$. \label{enum:fact on id}
\item If $\perte\in P_n$, then $t^2$ is idempotent and $\perte[t^2,e]\in P_n$.\label{enum:fact on t2 idpt}
\item If $\perte\in P_4$ and $g \in \Sn[4]$, then $\perte[t\sigma^g, e\sigma^g]\in P_4$.\label{enum:fact on P4}
\end{enumerate}
\end{lem}
\begin{proof}
\begin{enumerate}[left=0pt]
\item Let $t\in \TXn$ be an idempotent. Then clearly $t^3=t = t^2$, and hence $(t,t) \in P_n$.
\item Let $(t, e) \in P_n$ be such that $t^2=e$. Then we immediately obtain that $t = t^3 = tt^2 = te = e$.
\item If $t\in \Sn$, from $t^3=t$ we get that $t^2 = tt^{-1} = \id$, while $t\notin \Sn$ directly implies that $t^2\neq \id$. The second part follows from the fact that a permutation of order $2$ is a product of disjoint transpositions, and that $t\in U_n$ must fix at least one element of $\set{1,\dots,n}$ by Lemma~\ref{lem:counting}.
\item  If $t^k = t$, then we have that $(t^g)^k = (g^{-1}tg)^k = g^{-1}t^kg = g^{-1}tg = t^g$.
Conversely, if $(t^g)^k = t^g$, then we have $g^{-1}t^kg = g^{-1}tg$ which immediately gives us that $t^k = t$.
Finally, $g^{-1}tg = t^g = \id$ if and only if $t=g\,\id\,g^{-1}=\id$ as required. 
\item Let $\perte\in P_n$ and $g\in \Sn$. Since $t^3=t$, by the previous argument, we get that $(t^g)^3=t^g$.
Also, $t^ge^g = g^{-1}teg = g^{-1}eg = e^g$ and similarly $e^gt^g = e^g = (e^g)^2$, which shows that $\perte[t^g,e^g]\in P_n$. 
\item Let $\perte\in P_n$ so that $te = e$. Then, if $e=\id$ we get that $t = t\,\id = \id$, as required. 
\item Let $\perte\in P_n$. Then from $t^3 = t$ we directly obtain that $(t^2)^2 = t^3t= t^2$ and thus $t^2$ is an idempotent and lies in $U_n$ by point \ref{enum:fact idempotents in Un and Pn}.
Since $te = et =e$, we have that $t^2 e = t(te) = te = e$, and similarly $ et^2 = e$.
Hence $t^2 e = et^2 = e = e^2$, which shows that $\perte[t^2,e]\in P_n$. 
\item Let $\perte\in P_4$ and $g\in \mathcal{S}_4$. We first show that $\perte[t\sigma, e\sigma]\in P_4$ and then apply (5) (recalling the definition of $\sigma^g$)  to obtain the full result. Notice that if $t,e\in\TXn\priv \Sn$, then $t\sigma=e\sigma=c_4$, while if $t=e=\id$, then $t\sigma=e\sigma=\id$, which by part (1) shows in both cases that $\perte[t\sigma,e\sigma]\in P_4$. Otherwise, $t\in\Sn[4]$ and $e\neq\id$, and then $t\sigma = p_t$ and $e\sigma = c_4$. Since $\sigma$ is an endomorphism, we get that 
$(t\sigma)^3=t^3\sigma=t\sigma=p_t = t\sigma$. Moreover, as $p_t\in\Sn[4]$ fixes $4$ by definition, it follows that $p_t c_4=c_4p_t=c_4=c_4^2$  and we thus have $\perte[t\sigma,e\sigma]=\perte[p_t,c_4]\in P_4$.\qedhere
\end{enumerate}
\end{proof}

We now give some characteristics of the endomorphisms of $\TXn$ with rank at most three that will play an important part in the discussions to come. 
Elements of the proofs of the following two results given explicitly can be found scattered throughout the proof of Theorem~\ref{thm:endomorphisms} in~\cite{ST98}. 
They are only included here for convenience, and to prepare the reader for similar arguments to come. 

We refer in Corollary~\ref{cor:constraints of endomorphisms with e,t} to a semilattice ordering on idempotents, which we now explain. For any semigroup $S$ we order the idempotents $E=E(S)$ of $S$ by
$e\leq f$ if $ef=fe=e$. This is easily seen to be a partial order. If $Y$ is a commutative subsemigroup of $E$, then we refer to $Y$ as a 
{\em semilattice}. The reason for employing this terminology being that the order $\leq$ restricted to $Y$ is a 
{\em semilattice ordering} in the sense that for any $e,f\in Y$ the product $ef$ is the meet of $e$ and $f$ under the partial order.

\begin{cor}\label{cor:constraints of endomorphisms with e,t}
\begin{enumerate}
\item An endomorphism $\alpha \in \EndTn$ has rank $1$ if and only if $\alpha = \pee$ for some $e^2=e \in \TXn$. There is a one-to-one correspondence between the endomorphisms of rank $1$ and the one-element subsemigroups of $\TXn$.
\item An endomorphism $\alpha \in \EndTn$ has rank $2$ if and only if $\alpha = \pte$ for some $(t,e) \in P_n$ with $t^2 = t \neq e$. There is a one-to-one correspondence between the endomorphisms of rank $2$ and the two-element semilattices $\set{t,e} \subseteq \TXn$ with $e < t$.
\item An endomorphism $\alpha \in \EndTn$ has rank $3$ if and only if  $\alpha = \pte$ for some $(t,e) \in P_n$ with $t \neq t^2 \neq e$. There is a one-to one correspondence between the endomorphisms of rank $3$ and the three-element subsemigroups of $\TXn$ consisting of a two-element subgroup $\set{t,t^2}$ having identity element $t^2$, together with an adjoined zero $e$.
\item The map $\pte[t,\id]\in \EndTn$ if and only if $t=\id$.
\item If $\puf\in \EndTn$ for some $u,f\in \TXn$, then $\pte[u^2,f]$, $\pee[f]$ and $\ptctc[u]$ are also in $\EndTn$. 
\end{enumerate}
\end{cor}
\begin{proof}
Let $\alpha \in \EndTn$ be an element of rank at most three. Thus $\alpha = \pte$ for some $(t,e) \in P_n$ and it follows from the definition of $\pte$ that $\im\pte = \set{t, t^2, e}$.

For part (1) it is clear that $\alpha$ has rank $1$ if and only if $t=t^2=e$. In this case, since $e$ is idempotent it is clear that the image of $\pee$ is the trivial semigroup $\{e\}$. Conversely,  for each idempotent $e$
there is a unique endomorphism $\phi_{e,e}$ of rank $1$ with image $\{ e\}$.

For part (2), we note that it follows from Lemma~\ref{lem:facts on U} part (2) that $\pte$ having rank $2$ is equivalent to the condition that $t=t^2 \neq e$ (we have seen in Lemma~\ref{lem:facts on U} part (2) above that it is not possible to simultaneously have $t\neq t^2$ and $t^2 =e$, and since $e$ is idempotent it is also not possible to simultaneously have $t=e$ and $t^2\neq t$). In this case, the image of $\pte$ is $\{e,t\}$ and the relations $e=e^2=te=et$ and $t=t^2$ yield that this is a two element semilattice with $e<t$. Conversely, for each pair of distinct comparable idempotents $t,e$ with $t>e$ there is a unique endomorphism $\phi_{t,e}$ of rank $2$ with image $\{ t,e\}$. For part (3), it is clear that $\pte$ has rank $3$ if and only if $t \neq t^2 \neq e$. In this case $\pte$ has image $\{e,t,t^2\}$ and using the fact that $t^2$ is idempotent and $\perte[t^2,e]\in P_n$ by Lemma~\ref{lem:facts on U} part \ref{enum:fact on t2 idpt}, we obtain that $\set{t,t^2,e}$ is a subsemigroup of $\TXn$, where the idempotent $t^2$ acts identically on the left and right of $t$, and the idempotent $e$ acts as a left and right zero on all three elements. Conversely, for each two-element subgroup $\{ t,t^2\}$ where $t^2$ is idempotent, and each idempotent $e$ such that
$\{ t,t^2,e\}$ is a group with a zero $e$ adjoined, we have a unique endomorphism $\phi_{t,e}$ of rank $3$ with image $\{ t,t^2,e\}$. This proves parts (1)--(3). To see that (4) holds note that for the map $\pte[t,\id]$ to be in $\EndTn$, we require $\perte[t,\id]\in P_n$, which forces $t=\id$ by \ref{enum:fact on id} of Lemma~\ref{lem:facts on U}, and then $\pte[t,\id] = \phid$.
Finally, for part (5), we know that $\puf\in \EndTn$ if and only if $\perte[u,f]\in P_n$. But then using parts \ref{enum:fact on t2 idpt} and \ref{enum:fact idempotents in Un and Pn} of Lemma~\ref{lem:facts on U} we have that $u^2$ and $f$ are idempotents, and we therefore have that $u^2, f\in U_n$ as well as $(u^2, f) \in P_n$.
Hence $\pte[u^2,f]$, $\pee[f]$ and $\ptctc[u]$ satisfy all conditions to be endomorphisms.
\end{proof}

The next result is surprising in that singular elements of $\EndTn$ of rank no greater than 3 (so, all singular elements in the case $n\geq 5$) are entirely determined by their images. This has significant consequences later when we consider Green's relations.

\begin{lem}\label{lem:unicity of writing}
Let $\pte,\puf\in\EndTn$. 
Then $\pte=\puf$ if and only if \mbox{$\im\pte = \im\puf$} if and only if $t=u$ and $e=f$.
Further, if $n=4$ and $g,h\in\Sn[4]$ then $\sigma^g=\sigma^h$ if and only if $g=h$. 
\end{lem}
\begin{proof} 
Suppose that $\im\pte = \im\puf$ and consider the description of the images corresponding to the possible ranks of these maps as given in Corollary~\ref{cor:constraints of endomorphisms with e,t}. 
Clearly if $\pte$ and $\puf$ have rank $1$, then $t=e=u=f$. 
Suppose that they have rank $2$ so that $\im\pte=\set{t,e}$ and $\im\puf=\set{u,f}$ where $t$, $e$, $u$ and $f$ are all idempotents. 
Then, since their images are two element semilattices with $e<t$ and $f<u$, we get that $e=f$ and $t=u$.
Finally, if $\pte$ and $\puf$ have rank $3$, then $\set{t,t^2,e}=\set{u,u^2,f}$ where $t$ and $u$ are the only non idempotent elements, and $e$ and $f$ are the zeros, respectively, which together forces $t=u$ and  $e=f$. 
In all cases, we have shown that if $\im\pte=\im\puf$ then $t=u$ and $e=f$, and therefore $\pte=\puf$. This finishes the argument for the first statement.

If $\sigma^g=\sigma^h$ in $\EndTn[4]$, then by consideration of the images of these maps we find $4g=4h$ and $t^g = t^h$ for all $t\in\Sn[4]$ that fix $4$. The action on the transpositions
$(1\, 2)$ and $(1\, 3)$  forces $g=h$. The converse direction is clear.
\end{proof}

We also describe below the explicit multiplication of elements in $\EndTn$ for $n \neq 4$ as this will be a cornerstone of many later proofs.

\begin{cor}\label{cor:multiplication in ETn}
Let $g,h\in \Sn$ and $\perte, \perte[u,f]\in P_n$. Then we have the following compositions in $\EndTn$:
\begin{enumerate}
\item $\psi_g\psi_h = \psi_{gh}$;
\item $\psi_g\pte = \pte$;
\item $\pte\psi_g = \pte[t^g, e^g]$; and
\item $\pte\puf = \begin{cases} 
\puf &\mbox{if $t\in \Sn\priv \An$ and }e\neq \id,\\
\pte[u^2,f] &\mbox{if $t\in \An$ and }e\neq \id,\\
\pee[f] &\mbox{if $t\in \TXn\priv \Sn$ and }e\neq\id,\\
\ptctc[u] &\mbox{if }t=e=\id. \end{cases}$
\end{enumerate}
\end{cor}
\begin{proof}
All of the products are straightforward computations using the definition in Theorem \ref{thm:endomorphisms}; we only detail that for (4).
It is nonetheless worth noting that the map on the right-hand side of product (3) is well-defined and indeed belongs to $\EndTn$ by point \ref{enum:fact on conjugated pair} of Lemma~\ref{lem:facts on U}.

So consider $\perte, \perte[u,f]\in P_n$ so that $\pte, \puf\in \EndTn$, and let $s\in \TXn$.
Then:
$$ s\pte\puf = \begin{cases} t\puf &\mbox{if }s\in \Sn\priv \An, \\ t^2\puf &\mbox{if }s\in \An, \\ e\puf &\mbox{if }s\in \TXn\priv \Sn. \end{cases}$$ 
Recall from Lemma~\ref{lem:facts on U} part \ref{enum:fact on id} that if $e=\id$ then we must also have $t=\id$. Clearly, if $t=e=\id$, then we have that $t\puf = t^2\puf = e\puf = \id\puf = u^2$ so that $\pte\puf  = \ptctc[u]$ in this case. Thus in all remaining cases we may assume that $e \neq \id$ and hence $e \in \TXn \priv \Sn$.
If $t\in \TXn\priv\Sn$, then $t^2,e\in \TXn\priv \Sn$ so that $t\puf=t^2\puf=e\puf = f$.
Therefore we get that $\pte\puf = \pee[f]$ whenever $t\in\TXn\priv \Sn$. If $t\in\Sn$, then $t^2=\id$ by Lemma~\ref{lem:facts on U} part~\ref{enum:fact on t2 in Sn is id}, so that $t^2\puf = \id\puf=u^2$. 
In the case where $t\in\An$, we get that
$$ s\pte\puf = \begin{cases} t\puf = u^2 &\mbox{if }s\in \Sn\priv \An, \\ \id\puf = u^2 &\mbox{if }s\in \An, \\ e\puf=f &\mbox{if }s\in \TXn\priv \Sn, \end{cases}$$
which shows that $\pte\puf=\pte[u^2,f]$.
Otherwise, $t\in \Sn\priv\An$ and we obtain
$$ s\pte\puf = \begin{cases} t\puf = u &\mbox{if }s\in \Sn\priv \An, \\ \id\puf = u^2 &\mbox{if }s\in \An, \\ e\puf=f &\mbox{if }s\in \TXn\priv \Sn, \end{cases}$$
so that $\pte\puf=\puf$ in that case.
\end{proof}

Notice that for $n \neq 4$ the previous result encapsulates the multiplication table of $\EndTn$. In the next section we record the remaining products in the case $n=4$.

\subsection{Endomorphisms of rank 7 in \texorpdfstring{\boldmath{$\EndA[{\TXn[4]}]$}}{End(T4)}}
Let $D(4) = \{\sigma^g: g \in \Sn[4]\}$ denote the set of all endomorphisms of rank $7$ in $\EndTn[4]$. We recall from Theorem~\ref{thm:endomorphisms} the definition of the elements $\sigma^g$. We have   that $\Klein = \{\id, (1\,2)(3\,4), (1\,3)(2\,4), (1\,4)(2\,3)\}$ is a normal subgroup of $\Sn[4]$ and that each coset $\Klein s$ contains a unique element that fixes $4$, which we denote by  $p_s$. Then
$$s \sigma^g = \begin{cases}
p_s^g & \mbox {if } s \in \Sn[4],\\
c_{4g} & \mbox{if } s \in \TXn[4] \priv \Sn[4].
\end{cases}$$
Since $\Klein s = \Klein p_s$ it follows from the definition that for all $s \in \Sn[4]$ we have $p_s \sigma^g = s \sigma^g=p_s^g$. 
Since each of the maps $\sigma^g$ is an endomorphism of $\TXn[4]$ with the property that $\id$ is mapped to $\id$ 
(as clearly $\id=p_{\id}$), it is also easy to see that $p_{rs} = (rs)\sigma = r\sigma s\sigma = p_rp_s$, and so $(p_s)^{-1} = p_{s^{-1}}$. Morover, we note that $\{p_s: s \in \Sn[4]\} = \{h \in \Sn[4]: 4h=4\}$ and $h \sigma = h$ if and only if $4h=4$. These facts will be used without further comment.

\begin{lemma}
\label{lem:mult4}
	The sets $D(4)$ and $\EndTn[4] \priv D(4)$ are subsemigroups of $\EndTn[4]$. Moreover, we have the following compositions in $\EndTn[4]$:
 \begin{enumerate}
\item $\sigma^g\sigma^h = \sigma^{p_g h}$;
\item $\sigma^g\psi_h = \sigma^{gh}$;
\item $\psi_h\sigma^g = \sigma^{p_h g}$;
\item $\sigma^g\pte = \pte$; and
\item $\pte\sigma^g = \pte[t\sigma^g, e\sigma^g]$.
 \end{enumerate}
\end{lemma}

\begin{proof}
That $\EndTn[4] \priv D(4)$ is a subsemigroup follows from Corollary~\ref{cor:multiplication in ETn}. To see that $D(4)$ is a subsemigroup, consider  $\sigma^g, \sigma^h \in D(4)$. Using the fact that $\sigma$ is an endomorphism, for all $s \in \TXn[4]$ we have:
\begin{eqnarray*}
s \sigma^g\sigma^h &=& \begin{cases}
	(g^{-1}p_sg)\sigma^h & \mbox {if } s \in \Sn[4]\\
	c_{4g}\sigma^h & \mbox{if } s \in \TXn[4] \priv \Sn[4]
	\end{cases} = \begin{cases}
	h^{-1}(p_{g^{-1}}p_sp_g)h & \mbox {if } s \in \Sn[4]\\
	c_{4h} & \mbox{if } s \in \TXn[4] \priv \Sn[4]
	\end{cases},\\
 &=& s\sigma^{p_gh},
 \end{eqnarray*}
where we use the fact that both $p_s$ and $p_g$ fix $4$ and hence 
$p_s\sigma=p_s$ and $c_{4h} = c_{4p_gh}$. This shows that (1) holds, and hence that $D(4)$ is a subsemigroup.

For all $s \in \TXn[4]$ it is clear that
	$s\sigma^g\psi_h = h^{-1}s \sigma^g h = s\sigma^{gh}$, and so (2) holds. 
 Similarly, using the fact that $\sigma$ is an endomorphism $s\psi_h\sigma^g = (h^{-1}sh)\sigma^g = g^{-1}p_h^{-1}s\sigma p_hg= s\sigma^{p_h g}$ as given in (3).

For (4) note that for $s\in \mathcal{S}_4$ we have that $s\in \An[4]$ if and only if $p_s\in  \An[4]$. It follows that $s \sigma^g$ lies in $\Sn[4] \priv \An[4]$ (respectively, $\An[4]$, $\TXn[4] \priv \Sn[4]$) if and only if $s$ lies in $\Sn[4] \priv \An[4]$ (respectively, $\An[4]$, $\TXn[4] \priv \Sn[4]$).

Finally, for all $s \in \TXn[4]$ we have
	$$s\pte \sigma^g = \begin{cases} t\sigma^g & \mbox{ if } s \in \Sn[4] \priv \An[4]\\
	t^2\sigma^g & \mbox{ if } s \in \An[4]\\
	e\sigma^g & \mbox{ if } s \in \TXn[4] \priv \Sn[4].
	\end{cases}$$
Notice that since $\sigma^g$ is an endomorphism, we have that $t^2\sigma^g = (t\sigma^g)^2$. The result then follows from the fact that $\perte[t\sigma^g,e\sigma^g]\in P_4$ by \ref{enum:fact on P4} of Lemma~\ref{lem:facts on U}.
\end{proof}

\subsection{A decomposition via rank and type}
\label{sec:type}
The monoid $\EndTn$ can be partitioned in a convenient way by considering the following subsets of $\singn$:
\begin{align*}
E_3(n) &= \{\pte \in \EndTn \colon t \in \Sn \setminus \An, e\neq \id\},\\
A(n) &= \set{\pte\in \EndTn \colon t \in \An, t \neq \id \neq e},\\
B(n) &= \set{\pte\in \EndTn \colon t\in \TXn \setminus \Sn, t \neq t^2 \neq e \neq \id}, \\
E_2(n) &= \{\pte[\id,e] \in \EndTn\colon e\neq \id\},\\
C(n) &= \set{\pte\in \EndTn \colon t\in \TXn \setminus \Sn, t = t^2 \neq e\neq \id} \text{ and }\\
E_1(n) &= \{\pee \in \EndTn\}.
\end{align*}
For $n=4$ we also define $E_7(4) = \{\sigma^g: g \in \Klein\}\subsetneq D(4) =\{\sigma^g: g \in \Sn[4]\}$. To reduce notation, when it is clear from context we will suppress the dependence on $n$ and write simply $E_3, A, B, E_2, C, E_1, E_7$ and $D$. These subsets group together maps that share similar properties, such as the idempotents of a given rank:

\begin{lem}\label{lem:idempotents of EndTn}
For $k=1,2, 3,7$ the set $E_k$ consists of all the idempotents of rank $k$ in $\EndTn$. The set of idempotents of $\EndTn$ is therefore $$E(\EndTn) = \begin{cases}
\set{\varepsilon}\cup E_7\cup E_3\cup E_2\cup E_1 & \mbox { if } n=4,\\
\set{\varepsilon}\cup E_3\cup E_2\cup E_1 & \mbox { otherwise}. \\
\end{cases}$$ 
\end{lem}

\begin{proof}
Clearly, the only idempotent element of $\AutTn$ is $\psid$. 
It follows from the multiplication in Corollary \ref{cor:multiplication in ETn} that $\alpha = \pte\in \EndTn$ is idempotent if and only if either (i) $t$ is odd and $e \neq \id$ (in which case $\alpha \in E_3$), or (ii) $t = t^2\neq e$ and $t\in \Sn$ (in which case $t=\id$ and $\alpha \in E_2$) or (iii) $t=e$ from the last two cases of 4) in Corollary \ref{cor:multiplication in ETn} (in which case $\alpha \in E_1$). For $n=4$ it follows from Lemma \ref{lem:mult4} that $\sigma^g$ is idempotent if and only if $p_g = \id$ or, in other words, if and only if $g \in \Klein$.
\end{proof}

The remaining sets $A, B, C$ (together with $D \priv E_7$ in the case $n=4$) account for all the non-idempotent singular maps; the reasoning behind this grouping shall be made apparent shortly.

We remind the reader that a non-empty subset $I$ of a semigroup $S$ is an \emph{ideal} (respectively, a \emph{left} or \emph{right ideal}) if $SI\cup IS\subseteq I$ (respectively,  $SI\subseteq I$, or $IS\subseteq I$).

\begin{lem}\label{lem:partition of EndTn}
The endomorphism monoid $\EndTn$ can be written as:
$$\EndTn = \begin{cases}
\AutTn\cup D \cup \left(E_3 \cup B\right) \cup \left(E_2\cup C\right) \cup (E_1 \priv \{\phid\} \cup \{\phid\}) & \mbox{ if } n=4,\\
\AutTn\cup \left(E_3\cup A \cup B\right) \cup \left(E_2\cup C\right) \cup (E_1 \priv \{\phid\} \cup \{\phid\}) & \mbox{ otherwise,}
\end{cases}$$
where subsets containing endomorphisms of the same rank are bracketed together. For $n\geq 2$ this union is disjoint and the set $\EndTn \priv \AutTn$ is an ideal of $\EndTn$. 
\end{lem}
\begin{proof}
That $\EndTn$ is the union of the given sets follows from Theorem~\ref{thm:endomorphisms} and Corollary~\ref{cor:constraints of endomorphisms with e,t}, noting that the constraints of the given sets cover all eventualities and that the set $A$ is empty in the case of $n=4$. The fact that the bracketed expressions are the sets of endomorphisms of the same rank also follows from Corollary~\ref{cor:constraints of endomorphisms with e,t}. Each automorphism of $\EndTn$  has rank $n^n$ and hence the union is disjoint  for $n \geq 2$. 
\end{proof}

As we have seen in Corollary \ref{cor:multiplication in ETn}, there is an important distinction in the multiplicative behaviour of elements $\pte$ depending on where the element $t\in \TXn$ lies and whether $e=\id$. 
Since this will be of great importance to determine Green's relations and the ideal structure of $\EndTn$, we define the {\em type} of an endomorphism $\theta$ relative to the type of the underlying transformations associated with $\theta$.
\begin{defn}\label{def:types}
We say that $\theta \in \EndTn$ is of:
\begin{itemize}
\item \emph{group type} if $\theta\in \AutTn$;
\item \emph{exceptional type} if $n=4$ and $\theta\in D(4)$;
\item \emph{odd type} if $\theta \in E_3(n)$;
\item \emph{even type} if $\theta \in A(n) \cup E_2(n)$;
\item \emph{non-permutation type} if $\theta \in B(n) \cup C(n) \cup \left(E_1(n) \setminus \{\phid\}\right)$;
\item \emph{trivial type} if $\theta = \phid$. 
\end{itemize} 
Notice that the partition of $\EndTn$ given by Lemma \ref{lem:partition of EndTn} is therefore a partition into subsets of elements having the same rank \emph{and} type.
\end{defn}

The notion of type is a good one since this characteristic of a map is stable under multiplication by automorphisms.

\begin{lem}\label{lem:type stable under mult by automorphism}\label{lem:ABC only stable by G under right mult}
\begin{enumerate}
\item For  any $\phi_{t,e}\in \EndTn$ and $\psi_g\in \AutTn$ we have that
\[ \phi_{t,e}, \;\; \psi_g\phi_{t,e}=\phi_{t,e} \;\; \mbox{ and } \;\; \phi_{t,e}\psi_g=\phi_{t^g,e^g}\]
have the same type.
\item For any $\sigma^h\in \mathcal{E}_4$ and any $\psi_g\in \AutTn[4]$ we have that
\[ \sigma^h, \;\;  \psi_g\sigma^h=\sigma^{p_g h} \;\; \mbox{ and } \;\;  \sigma^h\psi_g=\sigma^{hg}\]
have the same (exceptional) type.
\item Let $n\neq 4$ and $\gamma\in \EndTn$. For any $\phi_{t,e}\in X$ where $X$ is one of $A$, $B$ or $C$, 
we have $\phi_{t,e}\gamma\in X$ if and only if $\gamma\in \AutTn$.
\end{enumerate}
\end{lem}
\begin{proof}
The proof of $(1)$ and $(2)$ follow immediately from Corollary~\ref{cor:multiplication in ETn} and Lemma~\ref{lem:mult4}, together with the observation that
conjugation in $\EndTn$ preserves the parity of elements in $\mathcal{S}_n$ and the rank of all elements. 

To show that $(3)$ holds, it only remains to show the converse. Let $\phi_{t,e}$ be in $A$, so that it is of even type and rank $3$.
By Corollary~\ref{cor:multiplication in ETn} we have that for any $\phi_{u,f}$ we have that $\phi_{t,e}\phi_{u,f}=\phi_{u^2,f}$  has rank at most $2$, so cannot lie in $A$.
Similarly, if $\phi_{t,e}$ is in $B$ or $C$, so that it has non-permutation type and is of rank $3$ or $2$, then $\phi_{t,e}\phi_{u,f}=\phi_{f,f}$ has rank $1$, so cannot lie in $B$ or $C$. 
\end{proof}

\begin{definition}
\label{notn:Xte}
For $\alpha \in \EndTn$ we define the {\em orbit} of $\alpha$ to be $\alpha\AutTn$. It is easy to see that all elements of a given orbit have the same rank and (by Lemma~\ref{lem:type stable under mult by automorphism}) the same type. In view of the decomposition given in Lemma \ref{lem:partition of EndTn} we note that each of the sets $\AutTn, E_3, A, B, E_2, C, E_1$ (and $D$ in the case $n=4$) is a union of orbits. For $\phi_{t,e} \in X$ where $X$ is one of $A$, $B$ or $C$, it will sometimes be convenient to write $X_{t,e}$ to denote the orbit of $\phi_{t,e}$, in order to easily recall the rank and type of elements in this orbit without specific mention of the corresponding properties of $t$ and $e$.
\end{definition}

For all $n \geq 5$, being of the same type is equivalent to acting in the same way by multiplication on the left on the singular part of $\EndTn$.
\begin{lem}\label{lem:equiv same type and products equal}
Let $n \geq5 $ and let $\alpha, \beta\in \singn$. Then $\alpha$ and $\beta$ are of the same type if and only if $\alpha\gamma = \beta\gamma$ for all $\gamma\in \singn$.
\end{lem}
\begin{proof}
Since $\alpha, \beta, \gamma\in \singn$, only (4) in Corollary \ref{cor:multiplication in ETn} is relevant.
It follows immediately from the description of this multiplication that if $\alpha$ and $\beta$ are of the same type, then $\alpha\gamma = \beta\gamma$ for all $\gamma \in \singn$.

Conversely, suppose $\alpha, \beta \in \singn$ are such that $\alpha\gamma=\beta\gamma$ for all $\gamma \in \singn$. If $\gamma = \phi_{u,f}\in \EndTn$ where $u = (2\,3) \neq \id = u^2 \neq f = c_1 \neq u$, then the maps $\puf$, $\pte[u^2,f]$, $\pee[f]$ and $\ptctc[u]$ are all distinct. From this, it is clear that if $\alpha$ and $\beta$ have different types, then we fall into into a different case for the multiplication and therefore $\alpha\gamma\neq \beta\gamma$, which gives us the equivalence.
\end{proof}

\begin{rem}\label{rem:bijections}
The authors of \cite{ST98} gave explicit formulae to count the number of endomorphisms of each rank. Since $\AutTn$ and $\Sn$ are isomorphic we have $\Card{\AutTn} = n!$. 
Meanwhile, it is clear that for any $e=e^2\in \EndTn$ we have that $(e,e)\in P_n$ and if $e\neq \id$ then also $(\id,e)\in P_n$.
Consequently, \mbox{$\Card{E_1(n)} = \Card{E_2(n)} + 1 = \Card{\set{e\in \TXn\colon e^2 = e}}$}. We shall not attempt to give formulae for the cardinality of the remaining sets in our partition of $\EndTn$, but it is useful to note that for $n \geq 5$ each set in this partition is non-empty. We conclude this section by recording examples to demonstrate this below; some of these examples will be utilised in later proofs.
\end{rem}

\begin{example}
\label{ex:nonemptyparts}
Let $n\geq 5$ and let $t, u, p,q, e, f\in \TXn$ be defined as follows:
\begin{gather*}
t = \pmap{1 & 2 & 3 & 4 & i_{\geq 5}\\1 & 3 & 2 & 1 & i}, \quad 
p = \pmap{1 & 2 & 3 & 4 & i_{\geq 5}\\2 & 1 & 4 & 3 & i}, \quad  
e = \pmap{1 & 2 & 3 & 4 & i_{\geq 5}\\5 & 5 & 5 & 5 & i},
\end{gather*}
\begin{gather*}
u = \pmap{1 & 2 & 3 & 4 & i_{\geq 5}\\1 & 1 & 1 & 4 & 4}, \quad q = \pmap{1 & 2 & 3 & 4 & i_{\geq 5}\\1 & 3 & 2 & 4 & i}, \quad 
f = \pmap{1 & 2 & 3 & 4 & i_{\geq 5}\\1 & 1 & 1 & 1 & 1}.
\end{gather*}
It is easy to verify that $p, q\in \Sn$, $p \in \An$, $q\in\mathcal{S}_n\setminus \mathcal{A}_n$ and $t, u, e, f\in \TXn\priv \Sn$ are such that:
\begin{itemize}
\item $e^2 = e \neq \id \neq f = f^2$;
\item $p^2 = \id = q^2$, $pe = e = ep$ and $qf = f = fq$ so that $\perte[p,e],\perte[q,f]\in P_n$;
\item $t^3 = t$ and $te = e = et$ so that $\perte\in P_n$.
\item $u^2 = u$ and $uf = f = fu$ so that $\perte[u,f]\in P_n$.
\end{itemize}
It then follows that the maps $\pte[p,e], \pte[q,f], \pte, \pte[\id, f]$ and $\pte[u,f]$ are endomorphisms of $\TXn$.
Moreover, it is clear from definition that $\pte[q,f] \in E_3(n)$, $\pte[p,e]\in A(n)$, $\pte\in B(n)$, $\pte[\id, f]\in E_2(n)$, $\pte[u,f]\in C(n)$, and $\pee[f]\in E_1\setminus\set{\phid}$.
\end{example}

\begin{remark}
We will make use of the preceding results repeatedly in the following sections where we describe the regular elements, minimal generating sets (and moreover, a presentation), Green's relations,  ideal structure and extended Green's relations for $\EndTn$. In order to state our results in their most general form, it will be convenient to assume that $n \geq 5$; this, of course, bypasses the case when $n=4$, where the structure of $\EndTn[4]$ is more complicated due to the additional maps of rank $7$, as well as some degenerate behaviour for $n \leq 3$ (where there are in some sense too few maps for the general behaviour to emerge). We will return to these special cases in Section~\ref{sec:degenerate}. 
\end{remark}

\section{Idempotents and regularity}
\label{sec:idempotent}
Throughout this section we assume that $n \geq 5$. In particular, this means that the set $\EndTn \setminus \AutTn$ of singular endomorphisms is equal to the set of endomorphisms of rank at most three:
$$\EndTn \setminus \AutTn= \{\pte: (t,e) \in P_n\} = \set{\pte \mid t,e\in \TXn \mbox{ with } t^3 = t \text{ and } te=et=e^2=e}.$$

\subsection{The left action of the endomorphism monoid on the singular part}
\label{sec:leftact}
For $(t,e) \in P_n$ the maps $\ptce$, $\pee$ and $\ptctc$ are all closely related to the map $\pte$.
Indeed their images are all contained in that of $\pte$ and by Corollary~\ref{cor:multiplication in ETn} they are all in $\EndTn\pte$.
Lemma~\ref{lem:equiv same type and products equal} allows us to define a notation facilitating this identification.

\begin{defn}\label{def:operations +,-,0}
Let $\alpha \in \EndTn \setminus \AutTn$. Then we define $\alpha^+$, $\alpha^-$ and $\alpha^0$ as:
\begin{itemize}
\item $\alpha^+ =\gamma \alpha$ for any $\gamma\in \singn$ of even type;
\item $\alpha^- =\gamma \alpha$ for any $\gamma\in \singn$ of non-permutation type; and
\item $\alpha^0 =\phid \alpha$ (that is, $\gamma\alpha$ for $\gamma\in \singn$ of trivial type).
\end{itemize}
Additionally, for $X\subseteq \singn$ and $\dagger\in \set{+,-,0}$, we define $X^\dagger$ to be the set $\set{\alpha^\dagger \colon \alpha\in X}$.
\end{defn}

\begin{rem}\label{rem:expression +,-,0}
Under this definition, we can now see that if $\alpha\in \singn$, then $\gamma\alpha \in \set{\alpha, \alpha^+, \alpha^-, \alpha^0}$ for any $\gamma\in \EndTn$ and thus $\EndTn\alpha = \set{\alpha, \alpha^+, \alpha^-, \alpha^0}$. 
Additionally, if we write $\alpha$ as $\pte$, then we have that
$$ \gamma\alpha = \begin{cases} \alpha = \pte &\mbox{if $\gamma$ has group type or odd type,} \\ \alpha^+ = \ptce &\mbox{if $\gamma$ has even type,} \\\alpha^- = \pee &\mbox{if $\gamma$ has non-permutation type,}\\ \alpha^0 = \ptctc &\mbox{if $\gamma$ has trivial type}. \end{cases}$$
\end{rem}

We now show how each of the sets involved in the decomposition of $\EndTn$ behave under the operations mapping $\alpha$ to $\alpha^+$, $\alpha^-$ or $\alpha^0$.

\begin{lem}\label{lem:operations on sets}
Let  $\alpha\in \singn$. Then
\begin{enumerate}
\item $1 = {\rm rank}(\alpha^0) = {\rm rank} (\alpha^-) \leq {\rm rank}(\alpha^+) \leq {\rm rank} (\alpha) \leq 3$;
\item $\alpha^+ = \alpha$ if and only if ${\rm rank}(\alpha) \leq 2$ if and only if $\alpha\in E_2\cup C\cup E_1$;
\item $\alpha^- = \alpha$ if and only if $\alpha^0 = \alpha$ if and only if ${\rm rank}(\alpha)= 1$ if and only if $\alpha\in E_1$;
\item $ E_3^+ \cup A^+ \subseteq  E_2 = E_2^+$, \; $B^+ \subseteq C = C^+$, \;and $E_1^+=E_1$;
\item $E_3^- \cup A^- \cup B^- \cup C^- \subseteq E_1 \setminus \set{\phid} = E_2^-$, \;and $E_1^-=E_1$;
\item $B^0\cup C^0\subseteq E_1\setminus \set{\phid} = (E_1\setminus \set{\phid})^0$, \;and $A^0 = E_3^0 = E_2^0 = \set{\phid} = \set{\phid}^0$.
\end{enumerate}

Consequently, we have $(\singn)^+ = E_2 \cup C \cup E_1$. and $(\singn)^- = (\singn)^0 = E_1$.
\end{lem}
\begin{proof}
Let $\alpha = \pte\in \EndTn$ for some $\perte\in P_n$.
Throughout this proof, we use the description of $\alpha^+$, $\alpha^-$ and $\alpha^0$ given in Remark \ref{rem:expression +,-,0}.

Part (1) follows immediately from Remark \ref{rem:expression +,-,0}. For part (2),
suppose first that $\alpha = \alpha^+$, which means that $\pte = \ptce$, so that $t=t^2$ and ${\rm rank}(\alpha) \leq 2$.
If $t\in \Sn$, this forces $t=\id$, and thus $\alpha = \phid[e]\in E_2$. 
Otherwise, $t\in \TXn\setminus \Sn$ and either $t\neq e$, which means that $\alpha\in C$, or $t = e$ and then $\alpha = \pee\in E_1$. 
Conversely, if $\alpha\in E_2\cup C\cup E_1$, then $t=t^2$ from which we have that $\alpha$ has rank at most $2$ and $\alpha^+ = \ptce = \pte = \alpha$.

For part (3), if $\alpha =\alpha^-$, then we have that $\pte = \pee$ which forces $t=e$, while if $\alpha = \alpha^0$, we require $\pte = \ptctc$ which implies that $t=t^2=e$. In both cases, this gives us that $\alpha\in E_1$, that is, has rank $1$. 
Conversely, if $\alpha\in E_1$, then $t = t^2=e$ and $\alpha^- = \pee =\alpha$ while $\alpha^0 = \ptctc = \pee =\alpha$.

For part (4), notice that we have already seen that $E_2^+=E_2$, $C^+=C$ and $E_1^+=E_1$. If $\alpha=\pte\in E_3\cup A$ so that $t \in \Sn$, then $\alpha^+ = \ptce = \phid[e]\in E_2$. On the other hand, if $\alpha = \pte\in B$, we have that $t^2 \neq  e$ and then $\alpha^+ = \ptce\in C$.

Similarly, for part (5) if $\alpha=\pte \in E_3 \cup A \cup B \cup E_2 \cup C$ then $e \neq \id$ and so $\alpha^- = \pte[e,e] \in E_1 \priv \{\phid\}$, whilst for each idempotent $e \neq \id$ we have that $\pte[\id, e]^{-} = \pte[e,e]$ and so $E_2^{-} = E_1 \priv \{\phid\}$. The remaining equality is given by part (3). Likewise, for part (6) if $\alpha=\pte \in B \cup C  \cup E_1 \priv \{\phid\}$ then $t^2 \neq \id$ and so $\alpha^0 = \pte[t^2,t^2] \in E_1 \priv \{\phid\}$, whilst for $\alpha=\pte \in E_3 \cup A  \cup E_2 \cup \{\phid\}$ we have $t^2 =\id$ and so $\alpha^0 = \phid$.
\end{proof}

The inclusions given in the previous lemma, as well as the representation of the action of the maps $\alpha\mapsto \alpha^x$ for $x\in \set{+,-,0}$ can be seen in Figure \ref{fig:J-order} alongside the depiction of the $\GJ$-order of the monoid $\EndTn$. 
The latter will be given explicitly in Proposition \ref{prop:principal ideals}.

\subsection{Idempotents and regular elements}

The set of idempotents of $\EndTn$, which we denote by $E=E(\EndTn)$,  has very nice properties. In particular, $E$ is a band. Before stating Proposition~\ref{cor:Ei are right-zero},  we give a little semigroup terminology and background. 
If we have a semilattice $Y$, then a semigroup $S$ decomposed into disjoint subsemigroups $S_i$, $i\in Y$, is a {\em semilattice $Y$ of the subsemigroups $S_i$} if $S_iS_j\subseteq S_{i\wedge j}$ for all $i,j\in I$.
It is a fact (see, for example, \cite[Theorem 4.4.1]{howie:1995}) that any band is a semilattice $Y$ of rectangular bands
$D_i$, $i\in I$, where a band is {\em rectangular} if it satisfies the identity $x=xyx$. A {\em right zero (resp. left identity)} for a semigroup $S$ is an element
$e$ such that $ae=e$ (resp.  $ea=a$) for all $a\in S$. A semigroup consisting entirely of right zeroes (equivalently, entirely of left identities) is called a  {\em right-zero semigroup}; clearly any such semigroup is a band, in fact, a special kind of rectangular band. In what follows, we use the term {\em chain} for a linearly ordered set
(which is, of course, a semilattice).

\begin{prop}\label{cor:Ei are right-zero}
\begin{enumerate} 
\item Each element of $E_3$ is a left identity of $\singn$.
\item Each element of $E_1$ is a right zero of $\singn$.
\item The minimal ideal of $\EndTn$ is $E_1$.
\item The set of all singular idempotents $E_3\cup E_2\cup E_1$ forms a left ideal of $\EndTn$.
\item The set of all idempotents $E=\{ \psid\}\cup E_3\cup E_2\cup E_1$  is a band, and forms a chain of right-zero semigroups.  
\end{enumerate}
\end{prop}
\begin{proof} Throughout, we bear in mind that $\psid$ is the identity of the monoid $\EndTn$. 

The fact that elements of $E_3$ are left identities for $\singn$ is immediate from 
Corollary~\ref{cor:multiplication in ETn}, as is the fact that elements of $E_1$ are right zeroes for $\singn$. Consideration of ranks immediately gives that $E_1$ is an ideal of $\EndTn$. Since $E_1$  is a right-zero semigroup it has no proper ideals, hence is the minimal ideal of $\EndTn$.  This verifies
(1)--(3).

For (4), let $\beta=\phi_{t,e}\in E_3\cup E_2\cup E_1$. Using Remark~\ref{rem:expression +,-,0} and Corollary~\ref{cor:multiplication in ETn}, we know that for any $\alpha\in \EndTn$
we have that\[\alpha\beta\in \{ \beta,\beta^+,\beta^-,\beta^0\}= \{ \phi_{t,e},\phi_{t^2,e},\phi_{e,e},\phi_{t^2,t^2}\}.\]
If $\beta\in E_3\cup E_2$ then $t^2=\id$; if $\beta\in E_1$ then $t=t^2=e$. It follows 
 that $\alpha\beta\in E\setminus\set{\psid}$ and hence $E_3 \cup E_2 \cup E_1$ is a left ideal. Moreover, for $\alpha, \beta \in E_k$ where $k=1,2,3$ it is straightforward to check that $\alpha \beta = \beta$ and hence each of these sets forms a right-zero semigroup and indeed $E$ is a band. The only thing remaining to check to see that (5) holds is that 
$E_2E_3\subseteq E_2$, and this is a now familiar calculation (also inherent in the above). 
\end{proof}

We have shown that the idempotents of $\EndTn$ form a {\em left regular band}, that is, a band that satisfies the identity $xyx=yx$. 

We now have all the tools necessary to describe  the regular elements of $\EndTn$. We recall at this point our assumption that  $n\geq 5$.

\begin{prop}\label{prop:regular elements}
 The set of all regular elements of $\EndTn$ is $\AutTn\cup E(\EndTn)$. Furthermore, this is a proper subsemigroup of $\EndTn$. In particular, $\EndTn$ is not regular.
\end{prop}
\begin{proof}
Clearly if $\alpha\in \EndTn$ is an automorphism or an idempotent, then it is a regular element.
The fact that $\AutTn\cup E(\EndTn)$ forms a (proper) subsemigroup is immediate from Lemma~\ref{lem:type stable under mult by automorphism}
and Proposition~\ref{cor:Ei are right-zero}. 

Conversely, let $\alpha\in \singn$ be a regular element of $\EndTn$ so that $\alpha\gamma\alpha = \alpha$ where $\gamma\in \EndTn$. Suppose for contradiction that $\alpha$ is not idempotent (that is, $\alpha\in A\cup B\cup C$) then it follows from Corollary \ref{cor:multiplication in ETn} part (2) that $\gamma \in \singn$. 
Note that $\alpha\gamma$ must be an idempotent left identity for $\alpha$ and hence of the same rank as $\alpha$. If $\alpha\in A\cup B$, then $\alpha$ has rank $3$; but since $\gamma \in \singn$ and $\alpha$ has even or non-permutation type, 
the rank of $\alpha\gamma$ is no greater than $2$, contradicting that $\alpha\gamma$ has the same rank as $\alpha$. Similarly, if $\alpha\in  C$, then $\alpha$ has rank $2$; but since $\alpha$ has non-permutation type, the rank of $\alpha\gamma$ is $1$.
The result follows.
\end{proof}

\begin{rem}\label{rem:regularcases} We will see in Section~\ref{sec:degenerate} that the {\em only} values of $n\in\mathbb{N}$ such that $\mathcal{E}_n$ is regular are
$n=1,2$.\end{rem}

\section{Generators and presentations}
\label{sec:gens}

Throughout this section we assume that $n \geq 5$. We demonstrate a minimal set of generators for $\mathcal{E}_n$ and, with respect to this set of generators, a presentation for $\mathcal{E}_n$.
Recall that for $\alpha \in \EndTn$, the orbit of $\alpha$ is $\alpha\AutTn$. All elements of the same orbit have the same rank, and so we shall speak of orbits of a given rank. We shall say that the orbit of $\alpha$  is {\em essential} if $\alpha$ cannot be generated by  elements of strictly greater rank. 

\begin{remark}
\label{rem:essential}
Clearly the orbit of each rank $3$ element, that is, each element in $E_3\cup A\cup B$, is essential. Meanwhile, there is a unique essential orbit of rank $1$, namely the orbit of $\phi_{\id,\id}$. Indeed,
any $\phi_{e,e}$, where $e\neq \id$ may be decomposed as a product $\gamma\phi_{\id,e}$ where $\gamma\in B$ and $\phi_{\id,e} \in E_2$ each have rank greater than $1$, and so the orbit is not essential. The orbit of $\phi_{\id,\id}$ is simply $\{ \phi_{\id,\id}\}$. 
\end{remark}

\begin{proposition}
Let $\Sigma \subseteq \EndTn$. Then $\Sigma$ is a generating set for $\EndTn$ if and only if $\Sigma$ contains a generating set for the group $\AutTn$ together with at least one element of each essential orbit. The minimal size of a generating set is therefore $3 + r_3(n) +r_2(n)$, where $r_3(n)$ denotes the number of essential orbits (equivalently, orbits)  of rank $3$ in $\EndTn$ and $r_2(n)$ denotes the number of essential orbits of rank $2$ in $\EndTn$.
\end{proposition}

\begin{proof}
Suppose first that $\Sigma$ is a generating set. It follows immediately from the multiplication given in Corollary \ref{cor:multiplication in ETn} that $\Sigma$ must contain a generating set for $\AutTn$. Clearly $\AutTn$ forms a single essential orbit, and $\Sigma$ contains at least one element of this orbit. By Remark \ref{rem:essential} we have that $\{\phid\}$ is the unique essential orbit of rank $1$, and since it is clear from Corollary \ref{cor:multiplication in ETn} that $\phi_{\id,\id}$ cannot be decomposed into a product of elements of $\EndTn \setminus \{ \phi_{\id,\id}\}$ it follows that any generating set must contain
$\phi_{\id,\id}$.

Assume that the rank of $\alpha\in \EndTn$ is $2$ or $3$, and the   orbit of $\alpha$ is essential. Since $\Sigma$ is a generating set we must have that $\alpha = \theta_1 \cdots \theta_m$ for some $m$ where $\theta_i \in \Sigma$. Since $\alpha$ is singular, it is clear that at least one of the generators $\theta_i$ is singular. Moreover, since elements of $\AutTn$ act as left identities on the singular elements, we may assume without loss of generality that $\theta_1, \ldots, \theta_{l}$ are all singular and $\theta_{l+1}, \ldots, \theta_m \in \AutTn$. Let $\gamma =\theta_1 \cdots \theta_{l-1}$ so that $\gamma\theta_l$ is in the orbit of $\alpha$ and in particular has the same rank as $\alpha$. We aim to show that the generator $\theta_l$ is in the orbit of $\alpha$ too.

 If $\alpha$ has rank $3$, then it follows from the facts that $3=|\im \alpha|=|\im \gamma\theta_l|\leq |\im \theta_l|\leq 3$ and  $\im\gamma\theta_l \subseteq \theta_l$  that 
$\im \gamma\theta_l=\im \theta_l$ and by Lemma~\ref{lem:unicity of writing} we have  $\gamma\theta_l =\theta_l$.

Suppose now  that $\alpha$ and hence $\gamma\theta_l$ have rank $2$. 
If $\theta_l$ has rank $3$ then as  $\gamma\theta_l$ has rank $2$ we must have that $\gamma\theta_l=\gamma'\theta_l$ for any representative $\gamma'$ of an essential orbit lying in $A$. But this would contradict our assumption that the orbit of $\alpha$ is essential. Thus
$\theta_l$ has rank $2$ and as above we have  $\gamma\theta_l =\theta_l$.

Conversely, suppose that $\Sigma$ is any set with the given properties. By Remark \ref{rem:essential} we have that $\phi_{\id,\id}\in \Sigma$. It is clear that we can generate all elements of $\AutTn$ (since we have a generating set for this finite group), and hence also all elements of the essential orbits (by definition). Since each rank $3$ orbit is essential this shows that we may generate all elements of rank at least $3$, which in turn allows us to generate all elements of rank $2$ (each element in a non-essential orbit by definition being obtained as products of elements of order greater than $2$). Finally, for each idempotent $e \in \TXn$ with $e \neq \id$ we have that $\pte[\id, e]$ is an endomorphism of rank $2$, and so for all $\gamma \in B$ of rank $3$ we have that $\gamma\pte[\id, e] = \pee$, which shows that all elements of rank $1$ can be generated.

Since $\AutTn$ is isomorphic to $\Sn$ it is clear that a minimal generating set for $\AutTn$ has two elements. Since all rank $3$ orbits are essential, whilst $\phid\AutTn = \{\phid\}$ is the unique rank $1$ essential orbit it then follows that a minimal generating set has $2 + r_3(n) +r_2(n) +1$ elements.
\end{proof}

We now give a sketch of how one might enumerate a minimal generating set for $\EndTn$. Clearly this amounts to finding a formula for 
$r_2(n)$ and $r_3(n)$. 
Each singular element  corresponds to a pair $(t,e)\in P_n$ and hence to a labelled coloured directed graph (one colour denoting the action of $t$ and one of $e$), with specific properties. 
Hence, each {\em orbit} of a singular element corresponds to an {\em unlabelled} coloured directed graph - and it must have precisely the following properties:
 \begin{itemize}
\item There are $n$ unlabelled vertices.
\item All edges are coloured either blue (corresponding to the action of $t$ and its conjugates) or red (corresponding to the action of $e$ and its conjugates).
\item There is exactly one red edge and exactly one blue edge leaving each vertex.
\item Looking only at the blue edges, recalling the constraint $t^3=t$  only, one can split the graph into components of the following two forms:
\begin{itemize}
\item a vertex with a loop, possibly with edges coming into it and
\item a two-cycle, possibly with edges coming in to the vertices.
\end{itemize}
\item Looking now at the red edges, recalling the constraints $te=e=et=e^2$, the red edges map all vertices to a subset of the blue-loop-vertices in such a way that:
\begin{itemize}
\item there is a red-loop at each vertex in the subset and 
\item all vertices of a blue component are mapped by a red edge to the same red-loop-vertex. 
\end{itemize}
 \end{itemize}
 
Counting the total number of orbits therefore amounts to counting the number of (non-isomorphic) unlabelled coloured directed graphs of this kind. Rank $3$ orbits correspond to the graphs containing at least one two-cycle. Rank $2$ orbits correspond to graphs where there is no two-cycle but the red and blue edges do not completely coincide.

\smallskip 
Having found minimal generating sets for $\EndTn$, we now give a presentation using those generators. First, we recall some details concerning  presentations for monoids: we must quotient free monoids by congruences.

Let $X$ be an alphabet  and denote by $X^*$ the free monoid on~$X$. 
The elements of $X^*$ are {\em words} over $X$, that is, finite products of elements of $X$, which may be referred to as {\em letters}.
We allow the empty product which we will denote by $1$.  Then $X^*$ becomes a monoid under concatenation of words, with $1$ being the identity. 
A {\em congruence} on $X^*$ is an equivalence relation that is compatible with the binary operation: if $\rho$ is such a congruence then
$X^*/\rho:=\{ [w]:w\in X^*\}$ becomes a monoid under $[v][w]=[vw]$ with identity $[1]$.  If~  $R\subseteq X^*\times X^*$, we denote by $R^\sharp$ the congruence on $X^*$ generated by $R$, that is, the smallest congruence on $X^*$ containing $R$. It is easy to see that $R^\sharp$ is the reflexive transitive closure of the set
$\{ (xvy, xwy): x,y\in X^*, (v,w)\in R\mbox{ or }(w,v)\in R\}$.   We say a monoid $S$ has \emph{monoid} \emph{presentation} $\langle X:R\rangle$ if~${S\cong X^*/R^\sharp}$ or, equivalently, if there is an epimorphism  $\theta:X^*\to S$  with kernel $R^\sharp=\ker \theta$. Here   the congruence $\ker\theta$
is defined by
$\ker\theta:=\{ (v,w):v\theta=w\theta\}$.  If $\theta$ is such an epimorphism, we say $S$ has \emph{presentation $\langle X:R\rangle$ with respect to $\theta$}.  A relation $(w_1,w_2)\in R$ will usually be displayed as an equation: $w_1=w_2$.  

Now let $ \Pi\cup \Sigma$ be a (minimal) generating set for $\EndTn$ consisting of a generating set $\Pi$ for $\AutTn$ together with a set $\Sigma$ containing exactly one element $\pte$ of each remaining essential orbit. We define
\[X_{\Pi}=\{ q_{g}: \psi_g\in \Pi\}\mbox{ and }X_{\Sigma}=\{ p_{t,e}: \phi_{t,e}\in \Sigma\}.\]
We choose and fix a single element of $X_{\Sigma}$ of odd (even, non-permutation, trivial) type and denote it by
$p^{od}$ ($p^{ev},p^{np},p^{tr}$) (supressing here the subscripts, and where $p^{tr}=p_{\id,\id}$).
Note that, by definition, $p^{od}\theta \in E_3$ has rank $3$. Since each rank $3$ orbit is essential, it is clear that we may also select $p^{ev}$ and $p^{np}$ so that their images under $\theta$ have rank $3$. With abuse of terminology, we may refer
to $p_{t,e}$ having (odd/even/non-permutation/trivial)   
type and rank $k$ (where $k \in \set{1,2,3}$) if $\phi_{t,e}$ does.  It follows from Remark~\ref{rem:essential} that $p^{tr}$ is the unique generator of rank $1$.

Let $X = X_{\Pi}\cup X_{\Sigma}$ and $\theta:X^*\rightarrow \EndTn$ be defined by
\[p_{t,e}\theta=\phi_{t,e}\mbox{ and } q_g\theta=\psi_g.\]
Let $R_\Pi$ consist of  
 any set of relations that yield $\langle X_{\Pi}:R_{\Pi}\rangle$ is a presentation of the symmetric group $\Sn$. We do not concern ourselves here with the `best' form of $R_{\Pi}$, as this is a well-trodden route. We aim to describe a set of relations $R_{\Pi \cup \Sigma}$ on $X$ that, when taken together with the relations $R_\Pi$, yields a presentation for $\EndTn$. In what follows, the underlying assumptions are that $q_g$ varies over $X_{\Pi}$ and $p_{t,e}$ varies over $X_{\Sigma}$. Define $F_{\Pi}(n)$  to be the smallest number 
 such that any element of $\mathcal{S}_n$ has a representation as a word of length at most $F_{\Pi}(n)$  in the generators $X_{\Pi}$. Let $R_{\Pi \cup \Sigma}$ consist of the following relations, where $m\leq F_\Pi(n)$:
\begin{itemize}\item[(R1)] $q_gp_{t,e}=p_{t,e}$ 
\item[(R2)]  $p_{t,e}q_{g_1}\ldots  q_{g_{m}}=p_{t,e}$ where $\phi_{t,e}\psi_{g_1}\ldots 
\psi_{g_{m}}=\phi_{t,e}$

\item[(R3)]  $p_{t_1,e_1}p_{t_2,e_2}q_{g_1}\ldots  q_{g_{m}}=p_{u_1,f_1}p_{u_2,f_2}$ where 
$\phi_{t_1,e_1}\phi_{t_2,e_2}\psi_{g_1}\cdots \psi_{g_{m}}=\phi_{u_1,f_1}\phi_{u_2,f_2}$ does not belong to an essential orbit
\item[(R4)] $p_{t,e}p_{u,f} = p_{t',e'} p_{u,f}$ where $p_{t,e}$ and $p_{t',e'}$ have the same type
\item[(R5)] $p^{od}p_{t,e} = p_{t,e}$  
\item[(R6)] $p^{ev}p_{t,e} = p_{t,e}$ where  $p_{t,e}$ has rank $2$
\item[(R7)] $p^{ev}p^{ev}p_{t,e}=p^{ev}p_{t,e}$
\item[(R8)] $p^{np}p^{ev}p_{t,e}=p^{np}p_{t,e}$,  $p^{ev}p^{np}p_{t,e}=p^{np}p_{t,e}$ and $p^{np}p^{np}p_{t,e}=p^{np}p_{t,e}$ 
\item[(R9)] $p^{ev}p^{tr}=p^{tr}$, $p^{np}p^{tr}=p^{tr}$ and   $p^{tr}p^{tr}=p^{tr}$ 
\item[(R10)]  $p^{tr}p^{np}p_{t,e}=p^{np}p_{t,e}$ and
 $p^{tr}p_{u,f} = p^{tr}$   where $p_{u,f}$ has odd or even type.
\end{itemize}
We now let \[R=R_\Pi\cup R_{\Pi\cup \Sigma}.\]

\begin{thm}\label{thm:pres} The monoid $\EndTn$ has presentation $\langle X:R\rangle$ with respect to $\theta$. 
\end{thm} 
\begin{proof} The morphism $\theta$ takes $X$ to a set of generators of $\EndTn$, so is an epimorphism. Let us denote $R^\sharp$ by $\sim$. Corollary~\ref{cor:multiplication in ETn} gives that
$\sim~\subseteq \ker\theta$; it remains to show the converse.

\begin{lem}\label{lem:noting}  
Let $w\in X^*$ and suppose that $w$ contains at least one letter from $X_{\Sigma}$. Then 
\[w\sim p_{t,e}q_{g_1}\cdots q_{g_m}\mbox{ or }w\sim p_{t_1,e_1}p_{t_2,e_2}q_{g_1}\cdots q_{g_m}\]
for some $p_{t,e},p_{t_1,e_1},p_{t_2,e_2}\in X_\Sigma$ and $q_{g_1}\cdots q_{g_m}\in X_{\Pi}$, for some $m\geq 0$.
\end{lem}
\begin{proof} By (R1) we have that $w\,\sim\, p_{t_1,e_1}\ldots p_{t_k,e_k}q_{g_1}\cdots q_{g_m}$
where $m\geq 0$. Either $k=1$, so that the first case holds. 
Otherwise,  using (R4) and (R5) we have that
$p_{t_1,e_1}\ldots p_{t_k,e_k}\,\sim\, p^1\ldots p^{k-1}p_{t_k,e_k}$ where $p^i\in \{ p^{ev},p^{np}, p^{tr}\}$ for
$1\leq i\leq k-1$. Finally,  (R7)--(R8) give that the second case holds.
\end{proof}

\begin{lem}\label{lem:notingrefined}  
Let $w\in X^*$ and suppose that $w$ contains at least one letter from $X_{\Sigma}$. Then the orbit of $w\theta$ is essential if and only if 
\[w\sim p_{t,e}q_{g_1}\cdots q_{g_m}\]
for some $p_{t,e}\in X_\Sigma$ and $q_{g_1}\cdots q_{g_m}\in X_{\Pi}$, for some $m\geq 0$.
\end{lem}
\begin{proof} If $w\sim p_{t,e}q_{g_1}\cdots q_{g_m}$ then it is clear the orbit of $w\theta$ is essential. Conversely, suppose that 
 the orbit of $w\theta$ is essential and the second case of Lemma~\ref{lem:noting} holds. Then 
$w\sim p_{t_1,e_1}p_{t_2,e_2}q_{h_1}\cdots q_{h_k}$.  We proceed by examining the rank of $w\theta$ to show that $p_{t_1,e_1}p_{t_2,e_2}\,\sim\, p_{t,e}$ for some $p_{t,e}$. 
Note that it follows from (R4) together with the fact that $p^{od}\theta, p^{ev}\theta$ and  $p^{np}\theta$ are assumed to have rank $3$ that we may assume without loss of generality that either $p_{t_1,e_1}\theta$ has rank $3$ or else $p_{t_1, e_1} = p^{tr}$.

If the rank of $w\theta$ is $3$, then this forces $\phi_{t_1,e_1}$ to be of odd type, and 
$p_{t_1,e_1}p_{t_2,e_2}\,\sim\, p_{t_2,e_2}$.

If the rank of $w\theta$ is $2$, then either $\phi_{t_1,e_1}$ is of odd type, and we proceed as above. Or, $\phi_{t_1,e_1}$ is of even type and
$\phi_{t_2,e_2}$ has rank $2$ or $3$. If $\phi_{t_2,e_2}$ has rank $2$, then $\phi_{t_1,e_1}\phi_{t_2,e_2}=\phi_{t_2,e_2}$ and also by (R4) and 
(R6) we have $p_{t_1,e_1} p_{t_2,e_2}\sim p_{t_2,e_2}$.  Otherwise, $\phi_{t_2,e_2}$ has rank $3$, and since by assumption $\phi_{t_1, e_1}$ has rank $3$, this contradicts the fact that $w\theta$ lies in an essential orbit.

If $w\theta$ has rank $1$ we note that, since this orbit is assumed to be essential and by Remark \ref{rem:essential} there is a unique essential orbit of rank $1$, we must have $w\theta = \phid$. Our aim then is to demonstrate that $p_{t_1, e_1}p_{t_2, e_2} \sim p^{tr} = p_{\id, \id}$. 
We shall consider all combinations of types for elements $p_{t_1, e_1}$ and $p_{t_2, e_2}$, showing that either $p_{t_1, e_1}p_{t_2, e_2} \sim p^{tr}$ as required, or else obtaining a contradiction to the fact the orbit of $w\theta$ is essential of rank $1$ (meaning that such a product does not lie in the unique rank $1$ essential orbit in the first place). 
If $\pte[t_1,e_1]$ is of odd type, then we proceed as in the previous cases. If $\pte[t_2,e_2]$ is of trivial type, then $\pte[t_1,e_1] \pte[t_2,e_2] = \pte[t_2,e_2]$ and by (R9) $p_{t_1,e_1} p_{t_2,e_2} \sim p_{t_2,e_2}$. We therefore assume in what follows that $p_{t_1,e_1}$ is not of odd type, and $p_{t_2,e_2}$ is not of trivial type. 
Notice in particular that this means that $p_{t_2, e_2}$ has rank $2$ or $3$ (since the only essential orbit of rank $1$ is the unique orbit of trivial type) and $e_2$ is therefore not the identity element (since if $e_2=\id$ we require that $t=\id$, and hence the element has rank $1$). 
If $\pte[t_1,e_1]$ is of even type then $\pte[t_1,e_1] \pte[t_2,e_2]$ has rank $2$, contradicting our assumption that $w\theta$ has rank $1$. If $\pte[t_1,e_1]$ is of non-permutation type, then $\pte[t_1,e_1] \pte[t_2,e_2] = \pte[e_2, e_2] \neq \phid$, contradicting the fact the orbit of $w\theta$ is essential. 
Finally we are left with the case that $\pte[t_1,e_1]$ has trivial type. Then if $\pte[t_2,e_2]$ has odd or even type we find $\pte[t_1,e_1] \pte[t_2,e_2] = \pte[t_2^2 ,t_2^2] = \phid$, giving $p_{t_1,e_1} p_{t_2,e_2} \sim p_{id,id}$ by (R10), whilst if $\pte[t_2,e_2]$ has non-permutation type we have $\pte[t_1,e_1] \pte[t_2,e_2] = \pte[t_2^2 ,t_2^2]  \neq \phid$ contradicting that the orbit of $w\theta$ is essential.
\end{proof}

We proceed with the proof of Theorem~\ref{thm:pres}. Let $w_1,w_2\in X^*$ and suppose that $w_1\theta=w_2\theta=\gamma$. 

If the rank of $\gamma$ is $n^n$, then all letters of $w_1$ and $w_2$ lie in $X_{\Pi}$ and $w_1\,\sim\, w_2$ using the relations in R$_0$.

Suppose now that the rank of $\gamma$ is no greater than $3$, so that each of $w_1,w_2$ has a letter from $X_{\Sigma}$.

 If
$w_1\sim p_{t,e}q_{g_1}\cdots q_{g_m}$ as in 
Lemma~\ref{lem:noting}, then $\gamma$ lies in an essential orbit. Then Lemma~\ref{lem:notingrefined} gives that $w_2\, \sim p_{u,f}q_{h_1}\cdots q_{h_k}$ for some
$k\geq 0$. Since we have chosen a unique generator corresponding to each essential orbit, we must have that $p_{t,e}=p_{u,f}$. 
It immediately follows that $\phi_{t,e}\psi_{g_1}\cdots \psi_{g_m}=\phi_{t,e}\psi_{h_1}\cdots \psi_{h_k}$,
giving that $\phi_{t,e}\psi_{g_1}\ldots \psi_{g_m}\psi_{h_k^{-1}}\ldots \psi_{h_1^{-1}}=\phi_{t,e}$. In
$\AutTn$ we have $\psi_{g_1}\ldots \psi_{g_m}\psi_{h_k^{-1}}\ldots \psi_{h_1^{-1}}=\psi_{d_1}\ldots \psi_{d_\ell}$ where
$\ell\leq n!$, so that using R$_0$ we have $q_{g_1}\ldots q_{g_m}q_{h_k^{-1}}\ldots q_{h_1^{-1}}\,\sim\, q_{d_1}\ldots q_{d_\ell}$.
Now
\[\begin{array}{rcll}
w_1&\sim& p_{t,e}q_{g_1}\cdots q_{g_m}\\
&\sim&p_{t,e}q_{g_1}\cdots q_{g_m}q_{h_k^{-1}}\ldots q_{h_1^{-1}}q_{h_1}\ldots q_{h_k}\\
&\sim&  p_{t,e}q_{d_1}\ldots q_{d_\ell}q_{h_1}\ldots q_{h_k}\\
&\sim& p_{t,e}q_{h_1}\ldots q_{h_k}&\mbox{ using (R2)}\\
&\sim& w_2,\end{array}\]
as required. 

The case where $\gamma$ does not lie in an essential orbit yields that $w_1\sim p_{t_1,e_1}p_{t_2,e_2}q_{g_1}\cdots q_{g_m}$ 
and $w_2\sim p_{u_1,f_1}p_{u_2,f_2}q_{h_1}\cdots q_{h_k}$, where $\phi_{t_1,e_1}\phi_{t_2,e_2}$ does not lie in an essential orbit. An
argument, similar to that displayed above, this time using (R3), gives that $w_1\,\sim\, w_2$ and finishes the proof.
\end{proof}

In Section~\ref{sec:related} we pose some related questions concerning presentations for
$\EndTn$ and other related monoids and semigroups.

\section{Green's relations and ideal structure}
\label{sec:Green}
Throughout this section we again assume that $n \geq 5$ so that the set $\EndTn \setminus \AutTn$ of singular endomorphisms is equal to the set of all endomorphisms of the form $\pte$ each having rank at most three. Here we turn our attention to Green's relations $\GR$, $\GL$, $\GD$, $\GH$ and $\GJ$.  It is well known that in a finite monoid, the $\GR$-class and the $\GL$-class of the identity coincide, and are hence equal to the $\GH$-class of the identity, which is the group of invertible elements (see for example \cite[Corollary 1.5]{St16}). In the case of $\EndTn$ is is easy to see this directly, using considerations of rank.

We first show that the $\GL$-classes of $\EndTn$ are singletons except for elements of $\AutTn$, which from the above form a single $\GL$-class.

\begin{prop}\label{prop:GL rel}
Let $\alpha, \beta\in \EndTn$.
\begin{enumerate}
\item If $\alpha \in \singn$ then the principal left ideal generated by $\alpha$ is 
$\EndTn \alpha = \{\alpha, \alpha^+, \alpha^-, \alpha^0\}$.
\item $\alpha\GL\beta$ if and only if $\im\alpha = \im \beta$ if and only if $\alpha = \beta$ or $\alpha,\beta\in \AutTn$.
\end{enumerate}
\end{prop}
\begin{proof}
(1) By Example \ref{ex:nonemptyparts} we know that $\EndTn$ contains an element of each type, and so Remark \ref{rem:expression +,-,0} applies to give that the principal left ideal is indeed $\EndTn \alpha = \{\alpha, \alpha^+, \alpha^-, \alpha^0\}$.

(2) Clearly if $\alpha=\beta$, then $\alpha\GL\beta$. 
On the other hand, if  $\alpha,\beta\in \AutTn$, then as remarked above $\alpha\GL\beta$.

Conversely, suppose that $\alpha\GL\beta$, so that $\alpha=\gamma\beta$ and $\beta=\delta\alpha$ for some $\gamma,\delta\in \EndTn$.
It follows that $\im\alpha = \im\beta$.
Clearly if this image is the whole of $\TXn$, then $\alpha,\beta\in \AutTn$.
Otherwise $\alpha, \beta\in \singn$, so that $\alpha=\pte$ and $\beta=\puf$ for some $\perte, \perte[u,f]\in P_n$ which gives us that $\alpha = \beta$ by Lemma~\ref{lem:unicity of writing}.
\end{proof}

\begin{remark}
Recall that if $X$ is one of $A$,$B$ or $C$, then $X$ is the union of orbits $X_{t,e}$ where $\perte\in P_n$ is such that $\pte\in X$. Moreover, it is easy to see that
\begin{align*}
\left(X_{t,e}\right)^+ &= \set{ (\pte[t^g,e^g])^+\colon g\in \Sn} 
 = \set{\pte[(t^g)^2,e^g]\colon g\in \Sn} \\
&= \set{\ptce\psi_g\colon g\in \Sn} = \ptce \AutTn = \pte^+ \AutTn,
\end{align*}
and we simply write this set as $X_{t,e}^+$. 
Notice that Lemma ~\ref{lem:operations on sets} gives $C_{t,e}^+ \subseteq C$ whilst $A_{t,e}^+ \not\subseteq A$ and $B_{t,e}^+ \not\subseteq B$.
\end{remark}

Proposition~\ref{prop:GL rel} shows that $\GL$ that is surprisingly restrictive. However, the $\GR$-classes of $\EndTn$ are much larger, as given by the following.

\begin{prop}\label{prop:GR rel}
Let  $\alpha\in \EndTn$. The principal right ideal of $\EndTn$ generated by $\alpha$ is equal to:
$$\alpha \EndTn = \begin{cases} 
	\EndTn & \mbox{ if } \alpha \in \AutTn,\\
	\EndTn \setminus \AutTn & \mbox{ if } \alpha \in E_3,\\
	A_{t,e} \cup E_2 \cup C \cup E_1   & \mbox{ if } \alpha =\pte \in A,\\
	E_2 \cup C \cup E_1  & \mbox{ if } \alpha \in E_2,\\
	B_{t,e} \cup E_1   & \mbox{ if } \alpha = \pte \in B,\\
	C_{t, e} \cup E_1  & \mbox{ if } \alpha = \pte \in C, or\\
	E_1  & \mbox{ if } \alpha \in E_1.
	\end{cases}$$
Consequently, $\AutTn$, $E_3$, $E_2$ and $E_1$ each consist of a single $\GR$-class, whilst each remaining $\GR$-class is the orbit of an element in $A$, $B$ or $C$.
\end{prop}
\begin{proof}
First note that if $\alpha\in \AutTn$, then $\alpha \EndTn = \EndTn$ while if $\alpha = \pte\in \singn$, then by Corollary~\ref{cor:multiplication in ETn} and Remark~\ref{rem:expression +,-,0}, we have that $\alpha \EndTn = \set{\pte[t^g, e^g] \colon g\in \AutTn} \cup \set{\gamma^x \colon \gamma\in\singn}$ where $\gamma^x$ is one of $\gamma$, $\gamma^+$, $\gamma^-$, or $\gamma^0$, depending on the type of $\alpha$.
For example, if $\alpha=\pte\in B$, then $\alpha \EndTn = \pte \AutTn \cup \set{\gamma^-\colon \gamma\in \singn} = B_{t,e}\cup \set{\pee[f]\colon f^2=f\in \TXn} = B_{t,e}\cup E_1$.

 Similar reasoning demonstrates that the principal right ideal generated by $\alpha$ is as stated in each of the remaining cases.
The description of the $\GR$-classes is then immediate.
\end{proof}

\begin{corollary}
\label{prop:GD,GJ,GH rel}
In $\EndTn$, we have that $\GH = \GL \subseteq \GD = \GJ = \GR$.
\end{corollary}
\begin{proof}
From Propositions \ref{prop:GL rel} and \ref{prop:GR rel}, one can directly see that $\GL \subseteq \GR$, and therefore $\GH = \GL \cap \GR = \GL$ while 
$\GD = \GL \circ \GR = \GL\vee \GR=\GR$. 
Additionally, since $\EndTn$ is finite, we have that $\GD = \GJ$ (see \cite[Proposition II 1.4]{howie:1995}).
\end{proof}

Corollary~\ref{prop:GD,GJ,GH rel} determines the $\GJ$-relation, and hence when two principal ideals coincide. Nevertheless, it is worthwhile recording their form.

\begin{prop}\label{prop:principal ideals}
Let  $\alpha\in \EndTn$. Then the principal two-sided ideal generated by $\alpha$ is
$$\EndTn\alpha \EndTn = \begin{cases}\EndTn & \mbox{ if } \alpha\in \AutTn,\\
\EndTn \setminus \AutTn &\mbox{ if } \alpha\in E_3,\\
A_{t,e}  \cup E_2 \cup  C \cup E_1 &\mbox{ if } \alpha = \pte \in A,\\
E_2 \cup  C \cup E_1 &\mbox{ if } \alpha\in E_2,\\
B_{t,e} \cup C_{t^2,e}  \cup E_1 &\mbox{ if }\alpha = \pte\in B,\\
C_{t,e}  \cup E_1 &\mbox{ if }\alpha = \pte \in C,\\
E_1 &\mbox{ if }\alpha\in E_1.\\
\end{cases}$$
\end{prop}
\begin{proof}
We have already given a description of the principal right ideals $\alpha \EndTn$ in Proposition~\ref{prop:GR rel}. It is clear that $\EndTn\alpha \EndTn = \alpha \EndTn$ for all $\alpha \in \AutTn$.  Suppose then that $\alpha \in \singn$. By Corollary \ref{cor:multiplication in ETn} and Remark \ref{rem:expression +,-,0} the description of the principal two-sided ideals $\EndTn \alpha \EndTn$ can therefore be found by taking the closure of the principal right ideal $\alpha \EndTn$ under the operations $\gamma \mapsto \gamma^+, \gamma \mapsto \gamma^-$ and $\gamma \mapsto \gamma^0$. By Lemma \ref{lem:operations on sets} we note that each ideal $\alpha \EndTn$ is closed under the latter two operations, and if $\alpha \not \in B$ then it is closed under all three operations, giving $\EndTn\alpha \EndTn = \alpha \EndTn$. For $\alpha = \pte[t^g, e^g] \in B_{t,e}$ we note that $\alpha^+ = \pte[(t^g)^2, e^g] = \pte[(t^2)^g, e^g] \in C_{t^2, e}$.
\end{proof}

In our considerations of Green's relations thus far, we have not mentioned that the relations $\GL,\GR$ and $\GJ$ are the equivalence relations associated with  pre-orders $\leq_{\GL}$,
$\leq_{\GR}$ and $\leq_{\GJ}$, defined by inclusions of left, right and  two-sided ideals, respectively.  

As an immediate consequence of the description of all the principal ideals of $\EndTn$ we can describe the $\GJ$-order of elements in $\EndTn$ as follows:
\begin{cor}\label{cor:J pre-order}
Let $\alpha, \beta\in \EndTn$. Then $\beta\leq_{\GJ}\alpha$ if and only if one of the following holds:
\begin{enumerate}
\item $\alpha\in \AutTn$;
\item $\alpha\in E_3$ and $\beta\in \singn$;
\item $\alpha = \pte\in A$, and $\beta \in A_{t,e}\cup E_2 \cup C\cup E_1$;
\item $\alpha \in E_2$ and $\beta\in E_2\cup C\cup E_1$;
\item $\alpha = \pte\in B$ and $\beta \in B_{t,e}\cup C_{t^2, e} \cup E_1$;
\item $\alpha = \pte \in C$ and $\beta\in C_{t,e}\cup E_1$; or 
\item $\alpha, \beta\in E_1$.
\end{enumerate}
\end{cor}

Using this, we can now see the structure of the $\GJ$-order of $\EndTn$ as laid out in Figure~\ref{fig:J-order}, where we have added how the maps $\alpha\mapsto\alpha^x$ for $x\in\set{+,-,0}$ act on the different components of $\EndTn$.

\begin{figure}[ht]
\centering
\includegraphics[scale=0.8, draft=false]{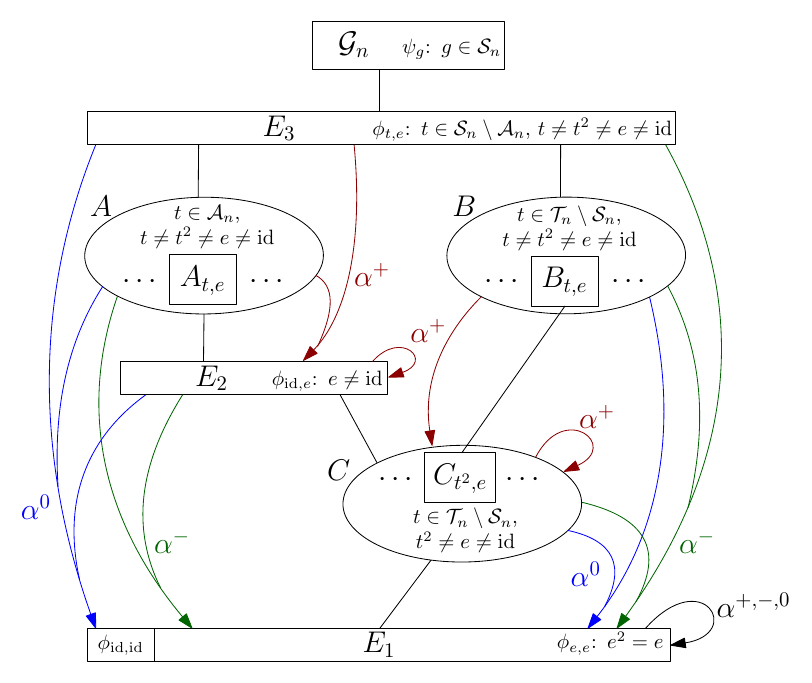}
\caption{The $\protect\GJ$-order of $\EndTn$ for $n \geq 5$. Each rectangle represents a $\GJ$-class, whilst the ovals represent the groupings of elements according to the sets $A, B, C$. Black lines indicate the $\GJ$-order, whilst directed lines indicate the action of the maps $\alpha \mapsto \alpha^x$ for $x \in \{+, -, 0\}$.}\label{fig:J-order}
\end{figure}

From the description of the principal ideals of $\EndTn$, we can also give an explicit formulation for each ideal of $\EndTn$.
In order to do so, we use the notation introduced in Definition~\ref{notn:Xte}.
Notice in particular that if $Y$ is a union of orbits  of elements of $B$, that is, $Y= \bigcup_{\pte\in B'} \pte \AutTn = \bigcup_{\pte \in B'} B_{t,e}$ for some $B'\subseteq B$, then
\begin{align*} 
Y^+ &= \bigcup_{\pte\in B'}\set{(\pte[t^g,e^g])^+\colon g\in \Sn} 
 = \bigcup_{\pte\in B'}\set{\pte[t^2,e]\psi_g\colon g\in \Sn} \\
&= \bigcup_{\pte\in B'} \ptce \AutTn = \bigcup_{\pte\in B'} C_{t^2,e}.
\end{align*}
Thus, $Y^+\subseteq C$ is again a union of orbits. Since each ideal is a union of principal ideals, we can call upon
Proposition~\ref{prop:principal ideals} to describe
 the ideals of $\EndTn$, as follows.

\begin{cor}\label{cor:ideals of EndTn}
Any ideal of $\EndTn$ takes one of the following forms:
\begin{enumerate}
\item $\EndTn$ (the only ideal containing automorphisms);
\item $\singn$ (the only proper ideal containing elements of odd type);
\item $X \cup Y \cup E_2 \cup C \cup E_1$ (ideals containing elements of even type, but no elements of odd type); or
\item $Y \cup Y^+ \cup Z \cup E_1$ (ideals containing only elements of trivial or non-permutation type),
\end{enumerate}
where the sets $X$, $Y$ and $Z$ are (possibly empty) unions of orbits taken from sets $A, B$ and $C$ respectively.
\end{cor}

\begin{remark}
We note that it follows from the multiplication given in Corollary~\ref{cor:multiplication in ETn} that every singular element squares to an idempotent, and so $\singn$ satisfies the identity $x^4=x^2$. Moreover, it is easy to see that in $\singn$ regular $\GD$-classes are aperiodic subsemigroups (that is, the subgroups are trivial) and in 
 $\EndTn$ the regular $\GD$-classes are subsemigroups. Another way of saying this is that $\singn$ belongs to the pseudovariety $\mathbf{DA}$ 
 and $\EndTn$ to the pseudovariety $\mathbf{DS}$ (see \cite{pin}). Since $\TXn$ does not belong to $\mathbf{DS}$ or $\mathbf{DA}$, this gives one way to see that $\TXn$ does not embed in $\EndTn$ or $\singn$ (although one can also show this more directly). Conversely, $\EndTn$ does not embed in $\TXn$: this can be seen by a quick counting argument, since for example the number of idempotents in $\EndTn$ is equal to $1 + |E_3| + |E_2| + |E_1|$ which exceeds the total number of idempotents $|E_1|$ in $\TXn$.
\end{remark}

\section{Extended Green's relations and generalised regularity properties}\label{sec:extended Greens}
We assume once more that $n \geq 5$ so that, as shown in Proposition~\ref{prop:regular elements}, the monoid $\EndTn$ is not regular. In order to better understand the structure of these endomorphism monoids, we turn to the extended Green's relations.
We now recall their precise definitions. The original sources for the $^*$-case is \cite{fountain:1982} and for the $^\sim$-case 
\cite{elqallali:1980} (see also \cite{lawson:1991}).

\begin{align*}
\alpha\GRs\beta \iff& \big(\gamma\alpha = \delta\alpha \ssi \gamma\beta =\delta\beta \quad \forall\gamma,\delta\in \EndTn\big),\\
\alpha\GLs\beta \iff& \big(\alpha\gamma = \alpha\delta \ssi \beta\gamma = \beta\delta \quad \forall\gamma,\delta\in \EndTn\big),\\
\alpha\GRt \beta \iff& \big(\eta\alpha=\alpha \Leftrightarrow \eta\beta = \beta \quad \forall \eta=\eta^2\in \EndTn \big),\\
\alpha\GLt \beta \iff& \big( \alpha\eta=\alpha \Leftrightarrow \beta\eta = \beta \quad \forall \eta=\eta^2\in \EndTn \big), \\
\GHs = \GLs \wedge \GRs = \GLs&\cap \GRs, \quad \phantom{and} \quad \GHt = \GLt \wedge \GRt = \GLt\cap \GRt, \\
\GDs = \GLs &\vee \GRs,\quad \phantom{and} \quad \GDt = \GLt \vee \GRt,\\
\alpha\GJs&\beta \iff J^*(\alpha) = J^*(\beta),\quad \text{and}\\
\alpha\GJt{}&\beta \iff \widetilde{J}(\alpha) = \widetilde{J}(\beta),
\end{align*}
where $J^*(\alpha)$ [resp. $\widetilde{J}(\alpha)$] is the smallest ideal containing $\alpha$ that is saturated by $\GLs$ and $\GRs$ \big[resp. by $\GLt$ and $\GRt{}$\big].

Before proceeding, we make some additional remarks. 
These relations come with the appropriate generalisation of the inclusions $\GH\subseteq \GR,\GL\subseteq \GD\subseteq \GJ$. 
As for Green's relations, $\GRs,\GLs,\GJs,\GRt,\GLt$ and $\GJt$ are the equivalence relations associated with certain preorders; we do not comment further on these. 
It is easily seen that $\GR\subseteq \GRs\subseteq \GRt$ and $\GL\subseteq \GLs\subseteq \GLt$. 
For any regular semigroup $S$ we have $\GR=\GRs= \GRt$ and $\GL=\GLs= \GLt$. In fact, if $a$ and $b$ are any regular elements of a semigroup $S$ and $a\GRt b$, then with $a=axa$ we have that $ax=(ax)^2$, so that $b=axb$; together with the converse we obtain $a\GR b$. A dual statement holds for $\GL$ and $\GLt$.
The converse is not true, as may be seen by (in an extreme case) considering a cancellative monoid that is not a group.
This gives  a hint at the idea that for non-regular semigroups, decomposing them using $^*$-classes or $^\sim$-classes may be useful. It is also worth remarking that, in general, $\GLs$ and $\GRs$ do not commute, so that $\GDs\neq \GLs\circ \GRs$; a similar statement holds for $\GLt$ and $\GRt$.

Before describing these relations on $\EndTn$, notice  from the above that  $\AutTn$ is contained in an $\GHs$-class and hence in an $\GHt$-class. 

\begin{lem}\label{lem:G is GRs-class} The group $\AutTn$ is an $\GRs$-class and an $\GRt$-class of $\EndTn$.
\end{lem}
\begin{proof} 
Since $\AutTn$ lies in an $\GRs$-class, and that $\GRs\subseteq \GRt$, it only remains to show that if $\beta\notin \AutTn$ then $\beta$ is not $\GRt$-related to $\psid$. 
But this is clear since $\beta=\phi_{t,e}\beta$ for any odd $\phi_{t,e}$, but certainly $\psid\neq \phi_{t,e}\psid$.
\end{proof}

For the description of the other $\GRs$-classes  we can show, using the type of maps in $\singn$, that it is sufficient to consider idempotents acting on the left in order to characterise the relation $\GRs$ for elements in $\singn$.

\begin{lemma}\label{lem:GRs depends only on idpt}
For any $\alpha,\beta\in \EndTn$, $\alpha\GRs\beta$ is equivalent to
$$\eta\alpha = \zeta\alpha \iff \eta\beta = \zeta\beta\quad \forall\eta, \zeta\in E(\EndTn).$$
\end{lemma}
\begin{proof}
Suppose first that $\alpha \GRs \beta$. By definition we have that for all $\gamma, \delta \in \EndTn$ we have $\gamma \alpha = \delta \alpha$ if and only if $\gamma \beta = \delta \beta$. In particular, this statement holds for all idempotents $\gamma, \delta \in \EndTn$. Conversely, suppose that for all $\eta, \zeta\in E(\EndTn)$ we have $\eta\alpha = \zeta\alpha \iff \eta\beta = \zeta\beta.$ If $\alpha \in \AutTn$ then it is clear that this condition is satisfied only if $\beta \in \AutTn$ (to see this, consider taking $\eta= \psid$, and $\zeta=\pte\in E_3$: their left action will agree on the singular elements, but will differ on the automorphisms), which by the previous result gives that $\alpha \GRs \beta$ in this case. 
Suppose then that $\alpha, \beta \in \EndTn\setminus \AutTn$,  and let  $\gamma, \delta \in \EndTn$ be such that $\gamma \alpha = \delta \alpha$.
Since there is an idempotent of each type, it follows from Lemma~\ref{lem:equiv same type and products equal} that $\eta \alpha = \zeta \alpha$ where $\eta$ is an idempotent of the same type as $\gamma$ and $\zeta$ is an idempotent of the same type as $\delta$, and hence $\gamma\beta = \eta\beta = \zeta\beta = \delta\beta$. A dual argument shows that  $\gamma \beta = \delta \beta$ implies $\gamma \alpha = \delta \alpha$, and hence $\alpha \GRs \beta$.
\end{proof}

Since $\GRs$ and $\GRt$ only depend on idempotents, we start by showing when a map admits a left identity.

\begin{lemma}\label{lem:left identities}
Let $\alpha\in \EndTn$. Then $\eta\alpha = \alpha$ for $\eta\in\EndTn$ if and only if one of the following holds:
\begin{enumerate}
\item $\alpha$ has rank $n^n$  and $\eta = \varepsilon$ (equivalently, $\alpha \in \AutTn$, and $\eta = \varepsilon$); 
\item $\alpha$ has rank $3$ and $\eta$ has group or odd type (equivalently, $\alpha\in E_3\cup A\cup B$ and $\eta\in \AutTn\cup E_3$);
\item $\alpha$ has rank $2$ and $\eta$ has group, odd, or even type (equivalently, $\alpha\in E_2\cup C$ and $\eta\in \AutTn\cup E_3\cup A\cup E_2$); or
\item $\alpha$ has rank $1$ and $\eta\in \EndTn$ (equivalently, $\alpha\in E_1$ and $\eta \in \EndTn$) .
\end{enumerate}
In particular, if $\eta$ is an idempotent then, $\eta\alpha = \alpha$ if and only if $\alpha \leq_{\GJ} \eta$.
\end{lemma}
\begin{proof}
It is clear that $\psid$ is the only left identity for $\alpha\in \AutTn$.
Assume now that  $\alpha\in \singn$. By consideration of rank, we can immediately determine the idempotent left identities. 
The result follows from Corollary~\ref{cor:J pre-order} and Lemma~\ref{lem:GRs depends only on idpt}.
\end{proof}

Furthermore, two idempotents $\eta$ and $\zeta$ satisfying $\eta\alpha = \zeta\alpha$ for a given map $\alpha\in \EndTn$ must lie above $\alpha$ in the $\GJ$ order, or have the same type. 
This is formally given by the following.

\begin{lem}\label{lem:pairs of idpt satisfying left-cancellation}
Let  $\alpha=\pte\in\singn$. Then $\eta\alpha = \zeta\alpha$ for some $\eta, \zeta\in E(\EndTn)$ if and only if one of the following happens:
\begin{itemize}
\item $\alpha\leq_{\GJ} \eta$ and $\alpha\leq_{\GJ}\zeta$; or
\item $\eta$ and $\zeta$ have the same type.
\end{itemize}
\end{lem}
\begin{proof} 
If  $\eta,\zeta\in E(\EndTn)$ are such that $\alpha \leq_{\GJ}\eta$ and $\alpha\leq_{\GJ}\zeta$, then $\eta\alpha = \alpha$ and $\zeta\alpha = \alpha$ by Lemma \ref{lem:left identities}, which shows that $\eta\alpha = \zeta\alpha$.
Similarly, if $\eta$ and $\zeta$ have the same type, then $\eta\alpha = \zeta\alpha$ by Lemma~\ref{lem:equiv same type and products equal} and the fact that $\psid$ is the only idempotent of group type.

For the converse, suppose that  that $\eta\alpha = \zeta\alpha$. 
If $\alpha\leq_{\GJ}\eta$, then $\eta\alpha = \alpha$ by Lemma~\ref{lem:left identities}.
Therefore $\zeta\alpha = \eta\alpha = \alpha$ so that $\alpha\leq_{\GJ}\zeta$, which corresponds to the first case.
Thus, we can now assume that $\alpha \nleq_{\GJ} \eta$ and $\alpha\nleq_{\GJ} \zeta$.
In particular, this means that $\alpha\notin E_1$ and so $t\neq e$   and that $\eta, \zeta\in E_2\cup E_1$.
Assume that $\eta$ is of even type and $\zeta$ is of trivial or non-permutation type, that is, $\eta\in E_2$ and $\zeta\in E_1$. 
Then $\eta\alpha = \zeta\alpha$ gives 
$$\ptce = \begin{cases} \ptctc& \text{if }\zeta = \phid, \\ \pee & \text{otherwise,}\end{cases}$$
which shows in either case that $t^2=e$, so that $t=e$ by Lemma~\ref{lem:facts on U} part~\ref{enum:fact on t2 is e forces t is e}, a contradiction. 
A similar contradiction arises when considering $\eta$ of trivial type and $\zeta$ of non-permutation type. 
Therefore $\eta$ and $\zeta$ must have the same type.
\end{proof}

We can now easily determine  the $\GRs$ and $\GRt$ relations.

\begin{prop}\label{prop:GRs and GRt rel}
Let  $\alpha, \beta\in \EndTn$. Then the following conditions are equivalent:
\begin{enumerate}\item $\alpha\GRs\beta$;
\item $\alpha\GRt\beta$;
\item $\alpha$ and $\beta$ are $\GR$-below the same idempotents;
\item   $\alpha$ and $\beta$ are \mbox{$\GJ$-below} the same idempotents;
\item $\alpha$ and $\beta$  have the same rank.
\end{enumerate}

Consequently, $\GRt = \GRs$ is a left congruence and the $\GRs$-classes of $\EndTn$ are $\AutTn$, $E_3\cup A\cup B$, $E_2\cup C$ and $E_1$.
\end{prop}

\begin{proof} If (1) holds, then so also does (2), since $\GRs\subseteq \GRt$. In any semigroup $S$, if
$ea=a$ for some $e,a\in S$, then clearly $a\leq_{\GR} e$.  On the other hand, if $a\leq_{\GR} f$ for some idempotent $f\in S$,
then from $a=fb$ for $b\in S^1$ we obtain $fa=ffb=fb=a$. Applying this to $\EndTn$ gives that (2) and (3) are equivalent.  Examination of Lemma~\ref{lem:left identities} now yields that  (3),(4) and (5) are equivalent. 

 It remains to show that if $\alpha\GRt \beta$ then  $\alpha\GRs \beta$. Suppose therefore  that $\alpha\GRt \beta$ and $\eta\alpha=\zeta\alpha$ where $\eta,\zeta$ are idempotent. Using Lemma~\ref{lem:pairs of idpt satisfying left-cancellation}, either  $\alpha\leq_{\GJ} \eta$ and $\alpha\leq_{\GJ}\zeta$ so that the same is true for $\beta$ and $\eta\beta=\beta=\zeta\beta$; or
$\eta$ and $\zeta$ have the same type, in which case certainly $\eta\beta=\zeta\beta$ by Lemma~\ref{lem:equiv same type and products equal}. Lemma~\ref{lem:GRs depends only on idpt} finishes the proof that
$\alpha\GRs\beta$, that is, (1) holds. That $\GRt$ is a left congruence then follows from the fact that in any semigroup $S$ the relation $\GRs$ is a left congruence. By point (5) the congruence classes are as given.
\end{proof}

We now turn our attention toward the relations $\GLs$ and $\GLt$.

\begin{proposition}\label{prop:GLt rel}
	The $\GLt$-classes of $\EndTn$ are $\EndTn\setminus (E_3\cup E_2 \cup E_1)$ and all the singletons $\set{\eta}$ where $\eta=\eta^2 \neq \psid$.
\end{proposition}
\begin{proof}
In order to show that $\alpha\GLt\beta$ for some $\alpha,\beta\in\EndTn$, we need to show that they have the same idempotents as right identities. 

If  $\alpha\in\AutTn$ then clearly $\alpha\eta = \alpha$ if and only if $\eta = \psid$.

If $\alpha\in A$ and $\eta=\puf\in E(\EndTn)$, then 
 by Lemma~\ref{lem:operations on sets} we have that $\alpha\eta = \eta^+\neq \alpha$, since $\eta^+$ has rank at most $2$.
 Thus the only idempotent right identity for $\alpha$ is $\psid$. A similar argument shows that $\psid$ is the only idempotent right identity for
 any element of $B\cup C$. 
 
 Turning our attention to the idempotents, we have remarked at the start of this section that idempotents are $\GLt$-related if and only if they are $\GL$-related, and hence by 
Proposition~\ref{prop:GL rel} if and only if they are equal. Since any idempotent is a right identity for itself, the result follows.
\end{proof}

\begin{remark}
Unlike the situation for $\GRs$ and $\GRt$, we find that $\GLs$ is a strictly smaller relation than $\GLt$. 
 Indeed, in general, the relation $\GLs$ on a semigroup $S$ is well-known to be a right congruence. However, it is easy to see that $\GLt$ is \emph{not} a right congruence. Indeed, taking $\alpha \in A$, $\beta \in B$ and $\pte \in E_3$ we find that $\alpha \GLt \beta$ whilst $\alpha\pte = \phid$ and $\beta\pte = \pee$ are distinct idempotents, and hence not $\GLt$-related.
\end{remark}

In what follows, for each $\perte\in P_n$ we write $\Fxte = \set{g\in \Sn\colon t^g = t \text{ and } e^g=e}$. 
Notice in particular that $\Fxe[\id,e] = \Fxe[e,e] $ for all $e^2=e$ and $\Fxe[\id, \id] = \AutTn$.
The next lemma is immediate.

\begin{lem}\label{lem:what fix does} For any $\pte\in \singn$ and $\psi_g\in \AutTn$ we have that
$\pte\psi_g=\pte$ if and only if $g\in \Fxte$.\end{lem}

We can now give the description of $\GLs$. 

\begin{prop}\label{prop:GLs rel}
Let  $\alpha, \beta\in \EndTn$. Then $\alpha\GLs\beta$ if and only if one of the following occurs:
\begin{itemize}
\item $\alpha, \beta\in \AutTn$;
\item $\alpha, \beta \in E_3\cup E_2\cup E_1$ and $\alpha = \beta$; or
\item $\alpha, \beta \in A\cup B\cup C$ are such that $\alpha = \pte$, $\beta=\puf$ have the same type and $\Fxte = \Fxe[u,f]$.
\end{itemize}
\end{prop}
\begin{proof}
Since $\GL\subseteq \GLs\subseteq \GLt$, we have that all elements of $\AutTn$ are $\GLs$-related, and that idempotents of $\EndTn$ distinct from $\psid$ form their own $\GLt$-class by Proposition \ref{prop:GLt rel} and thus they also form their own $\GLs$-class.
Therefore, it only remains to show that elements of $\AutTn$ cannot be $\GLs$-related to non-regular elements (that is, elements of $A\cup B\cup C$) and that two non-regular elements are $\GLs$-related if and only if they have the same type and are fixed by the same automorphisms of $\AutTn$.

To see first that no elements of $\AutTn$ can be $\GLs$-related to elements of $A\cup B\cup C$, consider $\alpha\in\AutTn$ and $\beta\in A\cup B\cup C$, and let $t=(1\,2), \, e=c_3\in\TXn$.
Since $t^2=\id$ and $\perte\in P_n$ we have $\pte, \phid[e],\pee\in\EndTn$.
Clearly $\alpha\pte, \alpha\phid[e]$ and $\alpha\pee$ are all distinct.
If $\beta\in A$, then we have that $\beta\pte = \phid[e] = \beta\phid[e]$, while if $\beta\in B\cup C$ we have $\beta\pte = \pee = \beta\phid[e]$.
Therefore elements of $\AutTn$ cannot be $\GLs$-related to elements of $A\cup B\cup C$.

From now on, we assume that $\alpha, \beta$ are non-regular (i.e. contained in $A \cup B \cup C$), and that $\alpha=\pte[v,k]$ and $\beta = \puf$ for some $\perte[v,k], \perte[u,f]\in P_n$.

Suppose that $\alpha\GLs\beta$ and consider the maps $\pte$, $\phid[e]$ and $\pee$ in $\EndTn$ as above.
To see that $\alpha$ and $\beta$ must have the same type, notice that if $\alpha\in A$ and $\beta\in B\cup C$, then $\alpha\pte=\phid[e]\neq \pee = \alpha\pee$ while $\beta\pte = \pee = \beta\pee$  which contradicts the fact that $\alpha\GLs\beta$.
Thus either $\alpha,\beta\in A$, or $\alpha,\beta\in B\cup C$ which shows that $\GLs$-related maps must be of the same type.
Lemma~\ref{lem:what fix does} gives that $\Fxe[v,k]=\Fxe[u,f]$. 

Conversely, assume that $\alpha$ and $\beta$ have the same type and that $\Fxe[v,k] = \Fxe[u,f]$.
By Lemma~\ref{lem:equiv same type and products equal} we have that $\alpha\eta = \beta\eta$ for all $\eta\in \singn$. 
Suppose now that $\psi_g,  \psi_h$ in $\AutTn$. If $\alpha\psi_g=\alpha\psi_h$, then $\alpha\psi_{gh^{-1}}=\alpha$.
It follows that  $gh^{-1}\in \Fxe[v,k]=\Fxe[u,f]$ and so $\beta\psi_g=\beta\psi_h$.
Finally, it is easy to see that if $\alpha\in A\cup B\cup C$, $\gamma\in\AutTn$ and $\delta\in \singn$, then
as the rank of $\alpha\gamma$ is the same as the rank of $\alpha$, but the rank of $\alpha\delta$ is strictly less than the rank of $\alpha$, we cannot have that 
$\alpha\gamma=\alpha\delta$.  
Using the  symmetry in  the arguments used above concludes the proof.
\end{proof}

\begin{defn}\label{defn:abundant} A semigroup $S$ is left [right] abundant if every $\GRs$-class [resp. $\GLs$-class] contains an idempotent, while it is left [right] Fountain if every $\GRt$-class [resp. $\GLt$-class] contains an idempotent.
\end{defn}

\begin{prop}\label{prop:abundant} For $n \geq 5$ the semigroup  $\EndTn$ is left abundant (and hence left  Fountain), right Fountain but not right abundant.
\end{prop}
\begin{proof} The result follows from Propositions~\ref{prop:GRs and GRt rel}, ~\ref{prop:GLt rel} and ~\ref{prop:GLs rel}.
\end{proof}

\begin{prop}
\label{prop:GDs and GJs rel}
The $\GDs$-classes of $\EndTn$ are $\AutTn$, $E_1$ and $\EndTn\setminus (\AutTn\cup E_1)$ and  
 further, $\GDs=\GJs$.
\end{prop}

\begin{proof}
By Propositions~\ref{prop:GRs and GRt rel} and~\ref{prop:GLs rel} we know that the elements of $\AutTn$ form a single $\GRs$-class and a single $\GLs$-class. Thus it follows that $\AutTn$ is also a $\GDs$-class.

 Let $\alpha \in E_1$ and supose that $\alpha \GLs \circ \GRs \delta$ for some $\delta \in \singn$. Then there exists $\eta \in \singn$ such that $\alpha \GLs \eta \GRs \delta$. By Proposition~\ref{prop:GLs rel} we find that $\eta = \alpha \in E_1$, and by Proposition~\ref{prop:GRs and GRt rel} we see that $\delta \in E_1$. It follows from this argument together with the fact that $E_1$ is an $\GRs$-class that $E_1$ is also a $\GDs$-class. 

Suppose then that $\alpha \in \EndTn \priv (\AutTn \cup E_1)$ and $\alpha \GDs \delta$. It follows from the above that $\delta \in \EndTn \priv (\AutTn \cup E_1)$. We aim to show that if $\alpha, \delta \in \EndTn\priv(\AutTn \cup E_1)$ then $\alpha \GDs \delta$. 

Notice that by our partition of $\EndTn$ we have that $\EndTn \priv (\AutTn \cup E_1) = (E_3 \cup A\cup B )\cup (E_2 \cup C)$, is the union of all elements of rank $3$ and all elements of rank $2$. Moreover, it follows from Proposition~\ref{prop:GRs and GRt rel} that if $\alpha, \delta$  either both have rank $3$ or both have rank $2$, then $\alpha \GRs \delta$, and hence $\alpha \GDs \delta$. We show that there is an element $\beta\in B$ of rank $3$ and $\gamma\in C$ of rank $2$ such that $\beta\GLs \gamma$, which will complete the description of $\GDs$.  

We recall from Example \ref{ex:nonemptyparts} that for 
\begin{gather*}
t = \pmap{1 & 2 & 3 & 4 & i_{\geq 5}\\1 & 3 & 2 & 1 & i}, \quad 
u = \pmap{1 & 2 & 3 & 4 & i_{\geq 5}\\1 & 1 & 1 & 4 & 4}, \quad  
f = \pmap{1 & 2 & 3 & 4 & i_{\geq 5}\\1 & 1 & 1 & 1 & 1},
\end{gather*}
we have $\beta :=\pte[t,f] \in B$ and $\gamma := \pte[u,f] \in C$, so that $\beta$ and $\gamma$ have the same type.  We show that $\Fxe[t,f] = \Fxe[u,f]$, so that $\beta \GLs\gamma$ by Proposition \ref{prop:GLs rel}.

Since $f = c_1$ we see that that $\Fxe[t,f] = \{ g \in \Sn: gt = tg \mbox{ and } 1g =1\}$. If $g \in \Fxe[t,f]$ we therefore have
$1=1g=4tg=4gt$ and hence (since $g$ is a permutation and $1g=1$) $4g=4$. For all $i \geq 5$ we have $ig=itg = igt$, that is, $ig$ is fixed by $t$ from which it follows that $ig \geq 5$ for all $i \geq 5$. Finally we have $2g=3gt$ and $3g=2gt$. Thus
$$\Fxe[t,f] = \{g, (2\,3)g: g \in \Sn, ig= i \mbox{ for } 1 \leq i \leq 4 \}.$$

Similarly $\Fxe[u,f] = \{ g \in \Sn: gu = ug \mbox{ and } 1g =1\}$. If $g \in \Fxe[u,f]$ we therefore have $1=1g=1gu=2gu=3gu$ and hence (since $g$ is a permutation and $1g=1$) we must have $\{2g, 3g\} = \{2,3\}$. Thus for all $i \geq 4$ we have that $ig \geq 4$. Moreover, since $u$ fixes $4g$ we must have $4g=4$ since $4$ is the only value distinct from $1$ that is fixed by $u$. It is then easy to see that 
$$\Fxe[u,f] = \{g, (2\,3)g: g \in \Sn, ig= i \mbox{ for } 1 \leq i \leq 4 \} = \Fxe[t,f].$$

We now look at the $\GJs$ relation. 

We know that $\GDs\subseteq \GJs$. Clearly for any $\beta\in E_1$ we have
$E_1=\EndTn \beta \EndTn$ is saturated by $\GDs$. It only remains to show that for 
some $\alpha\in  \EndTn\setminus (\AutTn\cup E_1)$ we have that $\EndTn\alpha\EndTn$ is saturated by $\GDs$. 
Taking 
any $\alpha\in E_3$ and calling upon Propositions~\ref{prop:GRs and GRt rel} and ~\ref{prop:GLs rel} yields the result.
\end{proof}

Finally,  the $\GDt$ and $\GJt$-relations are only composed of two classes: the minimal ideal $E_1$ and its complement $\EndTn\setminus E_1$.

\begin{prop}\label{prop:GDt and GJt rel}
 The $\GDt$-classes of $\EndTn$ are $\EndTn\priv E_1$ and $E_1$ and further,  $\GJt=\GDt$.
\end{prop}
\begin{proof}
Propositions~\ref{prop:GRs and GRt rel} and ~\ref{prop:GLt rel} immediately give us that there are two $\GDt$-classes, namely $\EndTn\priv E_1$ and $E_1$. Since $E_1$ is an ideal, it is a $\GJt$-class. The result follows.
\end{proof}

\section{The structure of \texorpdfstring{{$\EndA[\TXn]$ for $n \leq 4$}}{End(Tn) for n<=4}}\label{sec:degenerate}

In order to have clean statements with uniform proofs, in the previous sections we focussed on the case where $n \geq 5$. To complete the picture, in this section we describe the structure of $\EndTn$ in the cases where $n \leq 4$.  We note that the decomposition in terms of rank and type given in Lemma \ref{lem:partition of EndTn} can also be used to describe the structure of $\EndTn$ in these small cases, however, some of the sets turn out to be empty.  Indeed, $\EndTn[1] = \AutTn[1] = \{\varepsilon\}$ is a trivial group, and $\EndTn[2]$ decomposes as a disjoint union $\EndTn[2] = \AutTn[2] \cup E_2(2) \cup E_1(2)$ where $\AutTn[2] = \{\psid, \psi_{(1\,2)}\}$ is the automorphism group, whilst $E_2(2) = \{\pte[\id, c_1], \pte[\id, c_2]\}$ and $E_1(2) =\{\phid, \pte[c_1, c_1], \pte[c_2, c_2]\}$ are the idempotents of rank $2$ and $1$ respectively. We note in particular that these two semigroups (consisting of group elements and idempotents only) are regular. For $n=3$, the endomorphism monoid $\EndTn[3]$ decomposes as a disjoint union $\EndTn[3] = \AutTn[3] \cup E_3(3) \cup E_2(3) \cup C(3) \cup E_1(3)$.  The elements of $C(3)$ are not regular (note that the reasoning given in the proof of Proposition~\ref{prop:regular elements} also applies here). For the case $n=4$, recall that the endomorphism monoid $\EndTn[4]=\EndA[{\TXn[4]}]$ was described in Lemma~\ref{lem:partition of EndTn} as $\EndTn[4] = \AutTn[4]\cup D(4) \cup E_3(4) \cup B(4) \cup E_2(4)\cup C(4) \cup E_1(4)$ where each set is non-empty and $D(4)$ contains idempotents of rank $7$, namely the elements in the set $E_7(4) = \set{\sigma^g\colon g\in \Klein}$. 

In spite of these differences, certain properties turn out to be common to all endomorphism monoids $\EndTn$. Since their proofs are often akin to those presented in Sections~\ref{sec:idempotent}--\ref{sec:extended Greens} and these results could be obtained by direct enumeration using a computer program such as GAP, we will only highlight the main points that differ or require attention. A detailed version of these proofs can be found in the thesis of the second author \cite{grau}.
\begin{prop}
\label{prop:common}
Let $n$ be a natural number. In the endomorphism monoid $\EndTn$, the following statements hold: 
\begin{enumerate}
\item the set of all idempotents is a band, and forms a rank-ordered chain of
right zero semigroups;
\item the set of idempotents of rank $1$ is the minimal ideal of $\EndTn$;
\item $\GH = \GL \subseteq \GR=\GD=\GJ$;
\item $\GRs = \GRt$ and the classes are the sets of elements with the same rank.
\end{enumerate}
\end{prop}
\begin{proof}
These facts have already been explicitly proven for $n \geq 5$ in the previous sections, and it is straightforward to check that the details go through in just the same way for all $n \neq 4$. Indeed, for all $n \neq 4$, statements (1) and (2) follow in exactly the same way as in the proof of Proposition \ref{cor:Ei are right-zero}, whilst it is readily verified that the characterisation of the $\GL$ and $\GR$ relations given in Section \ref{sec:Green}  ($\alpha \GL \beta$ if and only if $\alpha = \beta$ or $\alpha, \beta \in \AutTn$; and 
$\alpha \GR \beta$ if and only if $\alpha, \beta \in E_k$ for some $k$ or $\alpha, \beta \in \AutTn$ or $\alpha \AutTn = \beta \AutTn$) also hold in these cases, from which statement (3) follows (since $\GL \subseteq \GR$). Part (4) is trivial in the case where $n=1,2$ (since these semigroups are regular) and can be proved in an entirely similar manner in the case $n=3$, using the fact that there is an idempotent of each type.

The case $n=4$ is a little different, since there are extra endomorphisms to consider. We use extensively the multiplication of elements described in Corollary~\ref{cor:multiplication in ETn} and Lemma~\ref{lem:mult4}.

For parts (1) and (2): It is clear from the multiplication in Corollary~\ref{cor:multiplication in ETn} 
 that $E_1(4)<E_2(4)<E_3(4)<\{\varepsilon\}$ is a chain of right-zero semigroups. That $E_1(4)$ is the minimal ideal can be seen by consideration of ranks. Notice that if $\alpha=\sigma^g\in E_7(4)$, then $p_g = \id$ and $\alpha\sigma^h = \sigma^h$, which means that elements of $E_7(4)$ are left identities for elements of $D(4)$, and hence in particular that $E_7(4)$ is also a right-zero semigroup. To complete the proof it suffices to verify that $E_7(4)E_k(4) \subseteq E_k(4)$ and $E_k(4)E_7(4) \subseteq E_k(4)$ for all $k\leq 3$, which follows immediately from Lemma~\ref{lem:mult4}.

For part (3): It is clear from the multiplication in Corollary~\ref{cor:multiplication in ETn} and Lemma~\ref{lem:mult4} that $D(4)$ and $\EndTn[4]\priv D(4)$  are subsemigroups of $\EndTn[4]$. Elements of $D(4)$ cannot be in the same $\GL$- or $\GR$-class as elements of $\EndTn[4]\priv D(4)$ since they do not have the same rank.
Moreover, from the multiplication one finds that two elements of $\EndTn[4]\priv D(4)$ will be $\GL$ (respectively, $\GR$) related in $\EndTn[4]$ if and only if they are $\GL$ (respectively, $\GR$) related in $\EndTn[4]\priv D(4)$, and one can describe the classes here in a similar manner to the case $n \neq 4$. Thus it suffices to describe how $D(4)$ splits into $\GL$ and $\GR$ classes. Since $\im \sigma^g = \set{t^g\colon t\in\Sn[4], 4t = 4}\cup\set{c_{4g}}$, it follows that $\im\sigma^g = \im\sigma^h$ if and only if $4g = 4h$. Thus $\alpha \GL \beta$ if and only if $\alpha = \beta$ or $\alpha, \beta \in \AutTn[4]$ or $\alpha=\sigma^g$, $\beta=\sigma^h$ both lie in $D(4)$ with $4g=4h$. Also, for all $g,h\in\Sn[4]$ we have that $p_g^{-1}h, p_h^{-1}g\in\Sn[4]$ with $\sigma^g\sigma^{p_g^{-1}h} = \sigma^h$ and $\sigma^h\sigma^{p_h^{-1}g} = \sigma^g$ which shows that $D(4)$ is an $\GR$-class. This yields that $\alpha \GR \beta$ if and only if $\alpha, \beta \in E_k(4)$ for some $1\leq k\leq 3$ or $\alpha, \beta \in \AutTn[4]$ or $\alpha,\beta\in D(4)$ or $\alpha \AutTn[4] = \beta \AutTn[4]$. Since $\GL\subseteq \GR$ we obtain the result.

For part (4): From the multiplication in Lemma~\ref{lem:mult4} we see that elements of $E_7(4)$ are left identities of $\EndTn[4]\priv \AutTn[4]$, but they are not left identities for elements of $\AutTn[4]$. Likewise, elements of $E_3(4)$ are left identities of $\EndTn[4]\priv \left(\AutTn[4]\cup D(4)\right)$, but they are not left identities for elements of $\AutTn[4]\cup D(4)$. 
It follows from this that $D(4)$ is an $\GRt$-class, and (since $\GR \subseteq \GRs \subseteq \GRt$) an $\GRs$-class. For the remaining elements, similar reasoning to that used in the proofs of Lemma~\ref{lem:G is GRs-class} and Proposition~\ref{prop:GRs and GRt rel} applies to show that the remaining classes follow the same pattern as before.    
\end{proof}

We recall from the discussion above that $\EndTn[3]$  is not regular, and so it makes sense to consider the extended Green's relations. 
\begin{prop}
In $\EndTn[3]$ the following statements hold:
\begin{enumerate}
\item $\GLs \subseteq \GRs = \GDs=\GJs = \GRt$ and the $\GRs$-classes are $\AutTn[3]$, $E_3(3)$, $E_2(3)\cup C(3)$ and $E_1(3)$ (i.e. elements of the same rank); 
\item $\alpha \GLt \beta$ if and only if $\alpha=\beta$  or $\alpha, \beta \in \AutTn[3] \cup C(3)$;
    \item $\alpha \GLs \beta$ if and only if $\alpha, \beta \in \AutTn[3]$ or $\alpha^2=\alpha = \beta=\beta^2$ or $\alpha = \pte, \beta = \puf$ both lie in $C(3)$ with $\Fxe[t,e] = \Fxe[u,f]$;
    \item the $\GDt$-classes are $\AutTn[3] \cup E_2(3)\cup C(3)$, $E_3(3)$ and $E_1(3)$;
    \item the $\GJt$-classes are $\EndTn[3]\priv E_1(3)$ and $E_1(3)$.
\end{enumerate} 
\end{prop}

\begin{proof} 
Since $\EndTn[3]$ contains an idempotent of each type as defined in Section \ref{sec:type}, almost all of the arguments given in Section \ref{sec:extended Greens} to go through verbatim. 
The only two notable differences concern the relations $\GDs$ and $\GDt$.
Indeed for $\GDs$ the proof of Proposition~\ref{prop:GDs and GJs rel} utilises an element of $B(n)$ to deduce that certain elements are $\GDs$-related, but since $B(3) = \emptyset$, this argument is not valid for $n=3$. The result however follows directly from the observation that $\GLs\subseteq \GRs$ so that $\GDs= \GRs$, while the proof that $\GDs = \GJs$ is just as before. 
In a similar manner in the case of $\GDt$, the proof of Proposition~\ref{prop:GDt and GJt rel} relies on the fact that elements of $A(n)$, $B(n)$ and $C(n)$ are $\GLt$-related in order to show that elements of $E_2(n)$, $E_3(n)$ and $\AutTn$ are $\GDt$-related. Since $A(3)=B(3)=\emptyset$ this argument does not hold for $n=3$. However, it is easy to see that elements of $E_3(3)$ form a single $\GDt$-class, while elements of $\AutTn[3]$ and $E_2(3)$ are $\GLtcGRt$-related. Thus the classes of $\GDt$ are as given in the statement.
Finally $\GDt\subseteq \GJt$, and since classes of $\GJt$ are ideals saturated by $\GLt$ and $\GRt$, it follows that the only two classes are $E_1(3)$ and $\EndTn[3]\priv E_1(3)$.
\end{proof}

Likewise, $\EndTn[4]$ is not regular, and the extended Green's relations can be described as follows:

\begin{prop}
In the endomorphism monoid $\EndTn[4]$, the following statements hold: 
\begin{enumerate}
\item $\alpha \GL \beta$ if and only if $\alpha = \beta$ or $\alpha, \beta \in \AutTn[4]$ or $\alpha=\sigma^g$, $\beta=\sigma^h$ both lie in $D(4)$ with $4g=4h$;
\item $\alpha \GR \beta$ if and only if $\alpha, \beta \in E_k(4)$ for some $1\leq k\leq 3$ or $\alpha, \beta \in \AutTn[4]$ or $\alpha,\beta\in D(4)$ or $\alpha \AutTn[4] = \beta \AutTn[4]$;
\item $\alpha \GLt \beta$ if and only if $\alpha=\beta$  or $\alpha, \beta \in \AutTn[4] \cup B(4) \cup C(4)$ or $\alpha=\sigma^g$, $\beta=\sigma^h$ both lie in $D(4)$ with $4g=4h$;
    \item $\alpha \GLs \beta$ if and only if $\alpha, \beta \in \AutTn[4]$ or $\alpha^2=\alpha = \beta=\beta^2$ or $\alpha=\sigma^g$, $\beta=\sigma^h$ both lie in $D(4)$ with $4g=4h$ or $\alpha = \pte, \beta = \puf$ both lie in $B(4)\cup C(4)$ with $\Fxe[t,e] = \Fxe[u,f]$;
    \item $\GDs =\GJs$ and the $\GDs$-classes are $\AutTn[4], D(4), E_3(4) \cup B(4) \cup E_2(4) \cup C(4)$ and $E_1(4)$;
    \item the $\GDt$-classes are $D(4)$, $E_1(4)$ and $\EndTn[4] \priv \left(D(4)\cup E_1(4)\right)$;
    \item the $\GJt$-classes are $\EndTn[4] \priv E_1(4)$ and $E_1(4)$. 
\end{enumerate} 
\end{prop}
\begin{proof} The proof of 
parts (1) and (2) is contained in that of Proposition~
\ref{prop:common}. 
For parts (3) and (4): Much of the proof from the previous section regarding $\GLt$ applies here, with the additional note that if $\alpha\in B(4)\cup C(4)$ and $\eta\in E_7(4)$, then $\alpha\eta\in E_1(4)$ so that the only right identity for elements in $B(4)\cup C(4)$ is $\psid$.
Now if $\alpha = \sigma^g\in D(4)$ and $\eta = \sigma^f\in E_7(4)$ are such that $\alpha\eta = \alpha$, it follows that $4g = 4f$. Thus $\sigma^g\GLt\sigma^h$ forces $4g=4h$ and thus $\sigma^g\GL\sigma^h$. This shows that for $\alpha, \beta\in D(4)$ we have $\alpha\GL\beta$ if and only if $\alpha\GLt\beta$ if and only if  $\alpha\GLs\beta$ (since $\GL \subseteq \GLs \subseteq \GLt$, and the description of $\GL$ has been given in the proof of Proposition \ref{prop:common}).
The proof that the remaining classes are as stated follows the reasoning of Proposition~\ref{prop:GLs rel}, noting that this is unaffected by the presence of elements in $D(4)$ and the absence of elements in $A(4)$. Indeed the presence of elements in $D(4)$ could only separate classes listed in Proposition~\ref{prop:GLs rel}, but these are either already $\GL$-classes, or they are classes in $B(4)\cup C(4)$. However, it is easy to see that if for some $\gamma\in D(4)$, $\delta\in \EndTn$ and $\alpha\in B(4)\cup C(4)$ one has $\alpha\gamma = \alpha\delta$ then $\beta\gamma=\beta\delta$ for all $\beta\in B(4)\cup C(4)$.
For part (5): The proof of $\GDs$ and $\GJs$ is exactly the same as in Proposition~\ref{prop:GDs and GJs rel} by considering the restrictions of elements $t,u$ and $f$ to the set $\set{1,2,3,4}$, that is, by using the elements
\begin{gather*}
t = \pmap{1 & 2 & 3 & 4 \\1 & 3 & 2 & 1}, \quad 
u = \pmap{1 & 2 & 3 & 4 \\1 & 1 & 1 & 4}, \quad \text{and} \quad  
f = \pmap{1 & 2 & 3 & 4 \\1 & 1 & 1 & 1},
\end{gather*}
since we then have that $\pte[t,f]\in B(4)$ and $\pte[u,f]\in C(4)$ are $\GLs$-related.

For parts (6) and (7): We have that $D(4)$ is a $\GDt$-class on its own since it is a single $\GRt$-class as well as the union of the $\GLt$-classes of its elements, and the other classes as given in the statement follow from the proof of Proposition~\ref{prop:GDt and GJt rel}.
However, since $\GDt\subseteq \GJt$ and the principal ideal generated by elements of $D(4)$ is not $\sim$-saturated, this forces elements of $\EndTn[4]\priv E_1(4)$ to form a single $\GJt$-class as required.
\end{proof}

\section{Related problems}\label{sec:related} As mentioned in the Introduction, this article is an opener in the discussion of the structure of endomorphism monoids of full transformation semigroups of various natural kinds. We present here a number of questions and possible further directions worthy of investigation.

\subsection*{Minimal presentations and counting problems} The presentation constructed in Section \ref{sec:gens} uses a minimal generating set, whose size has a nice combinatorial interpretation in terms of directed coloured graphs of certain types. Using the machinery of generating functions, it seems possible that one could give (recurrence) formulae for $r_3(n)$ (the number of orbits of rank $3$) and $r_2(n)$ (the number of essential orbits of rank $2$), and hence give a formula for the size $m(n)$ of a minimal generating set of $\EndTn$. A natural question in this regard is how $m(n)$ grows asymptotically, in particular as compared with the size of the endomorphism monoid $\EndTn$ itself. We note that some similar counting problems (concerning solutions to the system of equations $X^2 =X$, $Y^2 = Y$, $XY = YX$ in $\TXn$) have been considered in \cite{krattenthaler}. 

Whilst our generating set is minimal, we make no claim, however, as to the minimality of our \emph{presentation}. Indeed, the number of relations of the form (R2) and (R3) and the length of the words involved in these relation is dependent upon the chosen presentation $\langle X_\Pi: R_\Pi \rangle$ of the symmetric group $\Sn$. Specifically, it is dependent upon  $F_\Pi(n)$, that is, the smallest number 
 such that any element of $\mathcal{S}_n$ has a representation as a word of length at most $F_{\Pi}(n)$  in the generators $X_{\Pi}$.

\subsection*{Other questions concerning presentations}  We have considered a presentation for the monoid $\mathcal{E}_n$, based on a minimal set of generators. If we allowed ourselves to chose a generator from each orbit, then our relations would naturally simplify. Previous articles have considered presentations for the singular part of $\mathcal{T}_n$ \cite{east:2010,east:2013}. We could equally well consider (now semigroup) generators and presentations for $\singn$. 

\subsection*{Other semigroups of transformations}We have focussed on $\EndTn=\EndA[\TXn]$ as $\TXn$ is perhaps the most natural semigroup of \emph{finite} transformations.The automorphisms and endomorphisms of $\mathcal{PT}_n$ and $\mathcal{I}_n$ have already been described by different authors (see \cite[Chap. 7]{GM08}), where $\mathcal{PT}_n$ and $\mathcal{I}_n$ are, respectively, the
semigroups of partial (partial one-one) maps of a finite set, under composition of (partial) maps.
What can be said of the structure of $\EndA[{\mathcal{PT}_n}]$, $\EndA[{\mathcal{I}_n}]$ or $\EndA[{\mathcal{S}_n}]$? 

Moving away from transformation semigroups, one could consider related semigroups, such as Brauer monoids and partition monoids (see, for example, \cite{east:partition}), where here the endomorphisms have been determined \cite{mazorchuk}.

We have started with an unordered set
$\{ 1,\ldots,n\}$ and defined $\TXn$ to be all maps from this set to itself. There are ordered analogues of
$\TXn$ obtained by considering order preserving maps of
$\{ 1,\ldots,n\}$. The study of their endomorphisms has  begun in \cite{LF23a,LF23b}, and it would be interesting to develop these ideas further to describe the structure of the endomorphism monoids.

In another direction, one can drop the constraint of finiteness. Specifically, what of the endomorphism monoid of the full transformation monoid (and related semigroups) on an \emph{infinite} set, such as $\mathbb{N}$?

\subsection*{Semigroup vs. monoid endomophisms}  For any monoid $S$ we may consider the monoid of monoid endomorphisms (that is, semigroup endomorphisms also fixing the identity) or the monoid of semigroup endomorphisms. Here we have considered the latter for $\mathcal{T}_n$, but it would also be possible to adapt our results to consider the former.

\subsection*{Endomorphisms of endomorphism monoids of other algebras} We may regard regard $\TXn$ as the monoid of endomorphisms of an algebra with no operations. This is a degenerate case of a free algebra. One could consider endomorphisms of other endomorphism monoids of free algebra, perhaps starting with their automorphisms. A useful reference here is \cite{araujo}.

\subsection*{Iteration} The structure of $\mathcal{T}_n$ is quite rich: it forms a chain of $\mathcal{D}=\GJ$-classes, which are non-trivial, and it has elements of every rank.
It is, of course, the semigroup of endomorphisms of an algebra with no operations. By contrast, the structure of $\EndTn$ is rather thin - for example,
$\GL$ is trivial outside of the group of units. Will this make the structure of $\End(\EndTn)$ richer? In general, what can one say about the sequence of monoids
\[\mathcal{T}_n, \; \End(\mathcal{T}_n), \; \End(\End(\mathcal{T}_n)), \ldots ?\]

\end{document}